\documentclass{amsart}

\usepackage{amsmath,amsbsy,amsfonts,amsopn,amstext, bm}
\usepackage{euscript,amssymb,amsbsy,amsfonts,amsthm,latexsym,amsopn,amstext,amsxtra,euscript,amscd}
\usepackage{array,arydshln}
\usepackage{graphicx, verbatim}

\usepackage{longtable}

\usepackage{xy} 
\input xy \xyoption{all}

\usepackage{graphicx}

\usepackage{hyperref}

\hypersetup{
  colorlinks   = true, 
  urlcolor     = blue, 
  linkcolor    = blue, 
  citecolor   = red 
}

\newtheorem{thm}{Theorem} 

\newtheorem{prop}{Proposition} 
\newtheorem{lem}{Lemma}

\newtheorem{exa}{Example} 

\newtheorem{rem}{Remark}
\newtheorem{cor}{Corollary}

\newtheorem{clm}{Claim}

\newcommand\Z{\mathbb Z}
\newcommand\Q{\mathbb Q}
\newcommand\R{\mathbb R}
\newcommand\C{\mathbb C}

\def\deg{\mbox{deg }}

\newcommand\iso{\cong}

\def\iso{\cong}

\def\>{\rangle}
\def\<{\langle}

\newcommand\D{\Delta}                

\def\a{\alpha}

\usepackage[msc-links]{amsrefs}

\newcommand\Res{\mbox{Res}}

\title{Some remarks on the non-real roots of polynomials}

\author{Shuichi Otake}
\address{Department of Applied Mathematics\\ Waseda University \\ Japan }
\email{shuichi.otake.8655@gmail.com}

\author{Tony Shaska}
\address{Department of Mathematics and Statistics \\ Oakland  University \\ Rochester, MI, 48309. }
\email{shaska@oakland.edu}

\date{}                                           

\def\Gal{\mbox{Gal }}

\begin{document}

\begin{abstract}
Let $f \in \R( t) [x]$ be given by
$ f(t, x) = x^n + t \cdot g(x) $ 
and $\beta_1 <  \dots < \beta_m$ the distinct real roots of the discriminant 
$\D_{(f, x)} (t)$ of $f(t, x)$ with respect to $x$. Let $\gamma$ be the number of real roots of $g(x)=\sum_{k=0}^s t_{s-k} x^{s-k}$.
 For any $\xi > | \beta_m |$, if $n-s$ is odd then the number of real roots of $f(\xi, x)$ is $\gamma+1$,  and if  $n-s$ is even then the number of real roots of $f(\xi, x)$ is
  $\gamma$, $\gamma+2$ if $t_s>0$ or $t_s < 0$ respectively. A special case of the above result is constructing a family of totally complex polynomials which are reducible over $\Q$. 
\end{abstract}

\maketitle

\section{Introduction}\label{intro}

Let $f(x)\in \Q[x]$ be an irreducible polynomial of degree $n\geq 2$ and $\Gal (f)$ its Galois group over $\Q$. Let us assume that over $\R$, $f(x)$ is factored  as
\[ f(x) = a \, \prod_{j=1}^r   (x-\alpha_j) \, \prod_{j=1}^s (x^2 +a_i x+b_i).
\]
The pair $(r, s)$ is called the \textit{signature} of $f(x)$.  Obviously $\deg f = 2s+r$.   If $s=0$ then $f(x)$ is called \textit{totally real} and if $r=0$ it is called \textit{totally complex}.  By a reordering of the
roots we may assume that if $f(x)$ has $r$ non-real roots then
\[\a:=(1, 2) (3, 4)\cdots (r-1, r) \in Gal (f).\]
In \cite{b-sh} it is proved that if $\deg f =p$, for a prime $p$, and  $s$ satisfies
\[ s \, ( s \log s + 2 \log s + 3) \leq p \]
then $Gal (f)= A_p, S_p$. Moreover, a list of all possible groups for various values of $r$ is given for $p\leq 29$; see \cite[Thm.~2]{b-sh}.   There are some   follow up papers to \cite{b-sh}.  

In \cite{shimol} the author proves that if $ p \geq 4s+1$, then the Galois group is either $S_p$ or $A_p$. This improves the bound given in \cite{b-sh}. 
The author also studies when   polynomials with nonreal roots are solvable by radicals, which are consequences of  Table~2 and Theorem~2 in \cite{b-sh}. 
In \cite{otake} the author uses  Bezoutians of a polynomial and its derivative to  construct polynomials with real coefficients where the number of real roots can be counted explicitly. Thereby, irreducible polynomials in $\Q [x]$ of prime degree $p$ are constructed for which the Galois group is either $S_p$ or $A_p$.

In this paper we study a family of polynomials with non-real roots whose degree is not necessarily prime. Given a polynomial $g(x)= \sum_{i=0}^s t_i x^i$  and with $\gamma$ number of non-real roots we construct a polynomial $f(t, x) = x^n + t \, g(x)$ which has $\gamma, \gamma+1, \gamma+2$ non-real roots for certain values of $t\in \R$; see Theorem~\ref{thm5.2}.   The values of $t \in \R$ are given in terms of the Bezoutian matrix of polynomials or equivalently the discriminant of $f(t, x)$ with respect to $x$.  This is the focus of Section~\ref{Bezoutians} in the paper. 

While most of the efforts have been focusing on the case of irreducible polynomials over $\Q$ which have real roots, 
the case of  polynomials with no real roots is equally interesting.  How should an irreducible polynomial over $\Q$ with all non-real roots must look like?  What can be said about the Galois group of such totally complex polynomials?    In \cite{e-sh} is developed a reduction theory for such polynomials via the hyperbolic center of mass.   A special case of Theorem~\ref{thm5.2} provides a class of totally complex polynomials.  

\medskip

\noindent \textbf{Notation} For any  polynomial $f(x)$ by $\D_{f, x}$ we denote its discriminant with respect to $x$. If $f$ is a univariate polynomial then $\D_f$ is used and the leading coefficient is denoted by $\mbox{led} (f)$. 
Throughout this paper the ground field is a field of characteristic zero.

\section{Preliminaries}

Let $f_1(x)$, $f_2(x)$ be polynomials over a field $F$ of characteristic zero and, let $n$ be an integer which is greater than or equal to $\max \{ \mathrm{deg} f_1, \mathrm{deg} f_2 \}$. Then, we put
\begin{align*}
B_{n}(f_{1},f_{2}) :&= \frac{f_{1}(x)f_{2}(y)-f_{1}(y)f_{2}(x)}{x-y}= \sum_{i,j=1}^{n} \alpha _{ij}x^{n-i}y^{n-j}\in F[x,y], \\ 
M_{n}(f_{1},f_{2}) :&= (\alpha_{ij})_{1 \leq i,j \leq n}.
\end{align*} 
The matrix $M_{n}(f_{1},f_{2})$ is called the \textit{Bezoutian} of $f_1$ and $f_2$. Clearly, $B_n(f_1, f_1)=0$ and hence $M_n(f_1, f_1)$ is the zero matrix.   The following properties hold true; see \cite[Theorem 8.25]{fuh}  for details. 

\begin{prop}\label{prop5.1}
The following are true:
\begin{enumerate}
\item[$(1)$] $M_n(f_1, f_2)$ is an $n \times n$ symmetric matrix over $F$. 

\item[$(2)$] $B_{n}(f_{1},f_{2})$  is linear in $f_{1}$ and $f_{2}$, separately. 

\item[$(3)$] $B_{n}(f_{1},f_{2})=-B_{n}(f_{2},f_{1})$.   
\end{enumerate}
\end{prop}


When $f_2=f_1^{\prime}$, the formal derivative of $f_1$ (with respect to the indeterminate $x$), we often write $B_n(f_1):=B_{n}(f_{1},f_{1}^{\prime})$.
From now on, for any  degree $n\geq 2$ polynomial $f(x) \in \R[x]$ we will denote by $M_n (f):=M_n (f, f^\prime)$ as above. The matrix $M_n(f)$ is called the \textbf{Bezoutian matrix} of $f$.

\begin{rem}
It is often the case that the matrix $M^{\prime}_{n}(f_{1},f_{2})=(\alpha_{ij}^{\prime})_{1\leq i,j \leq n}$ defined by the generating function
\begin{align*}
B^{\prime}_{n}(f_{1},f_{2}) :&= \frac{f_{1}(x)f_{2}(y)-f_{1}(y)f_{2}(x)}{x-y}= \sum_{i,j=1}^{n} \alpha _{ij}^{\prime}x^{i-1}y^{j-1}\in F[x,y]
\end{align*} 
is called the Bezoutian of $f_1$ and $f_2$. But no difference can be seen between these two definitions as far as we consider the corresponding quadratic forms
\[
\sum_{i,j=1}^n\alpha_{ij}x_ix_j \hspace{3mm} \text{and} \hspace{3mm} \sum_{i,j=1}^n\alpha_{ij}^{\prime}x_ix_j.
\]
In fact, these two quadratic forms are equivalent over the prime field $\Q$ $(\subset F)$ since we have
$M_n^{\prime}(f_1, f_2)={}^tJ_nM_n(f_1, f_2)J_n$,
where
\[
J_n
=
\left[
\begin{array}{cccc}
\multicolumn{1}{c}{\raisebox{-7pt}[0pt][0pt]{\huge $0$}} & & & 1 \\
& & 1 & \\
& \rotatebox[origin=c]{305}{\vdots} &  & \\
1 & & & \multicolumn{1}{c}{\raisebox{0pt}[0pt][0pt]{\huge $0$}}  
\end{array}
\right]
\]
is an $n \times n$ anti-identity matrix. This implies that above two quadratic forms are equivalent over $\Q$ or more precisely, over the ring of rational integers $\Z$. 
\end{rem}

Let $f(x)\in \R[x]$ be a degree $n\geq 2$    polynomial which  is given by
\[ f(x) = a_0 + a_1 x + \dots + a_n x^n \]
Then over $\R$ this polynomial is   factored    as
\[ f(x) = a \, \prod_{j=1}^r   (x-\alpha_j) \, \prod_{j=1}^s (x^2 +a_i x+b_i).
\]
for some $\alpha_1, \dots , \alpha_r \in \R$ and $a_i, b_i, a \in \R$.

The pair of integers  $(r, s)$ is called the \textbf{signature} of $f(x)$.  Obviously $\deg f = 2s+r$.   If $s=0$ then $f(x)$ is called \textbf{totally real} and if $r=0$ it is called \textbf{totally complex}.
Equivalently the above terminology can be defined for binary forms $f(x, z)$. 

Throughout this paper, for a univariate polynomial $f$, its discriminant will be denoted by $\D_f$. For any two polynomials $f_1(x)$, $f_2 (x)$ the resultant with respect to $x$ will be denoted by $\Res (f_1, f_2, x)$.
We notice the following elementary fact, its proof is elementary and we skip the details.

\begin{rem}
For any polynomial  $f (x)$, the determinant of the Bezoutian is the same as the discriminant up to a multiplication by a constant. 
More precisely,    
\[   
\D_f = \frac 1 {\mbox{led} (f)} \det M_n (f ),
\] 
where $\mbox{led} (f)$ is the leading coefficient of $f(x)$. 
\end{rem}

If $f(x)\in \Q[x]$ is irreducible and its    degree   is a prime number, say $\deg f = p$, then there is enough known for the Galois group of polynomials with some non-real roots; see \cite{b-sh}, \cite{shimol}, \cite{otake} for details.  
If the number of non-real roots is "small" enough with respect to the prime degree $\deg f =p$ of the polynomial, then the Galois group is $A_p$ or $S_p$. Furthermore, using the classification of finite simple groups one can provide a complete list of possible Galois groups for every polynomial of prime degree $p$ which has non-real roots; see \cite{b-sh} for details. 


On the other extreme are the polynomials which have all roots non-real.  We called them above, totally complex polynomials. 
We have the following:

\begin{lem}
The followings are equivalent:

     i)  $f(x) \in \R[x]$ is totally complex
     
     ii)  $f (x)$ can be written as 
\[ f(x) = a \displaystyle \prod_{i=1}^n f_i \]
where $f_i = x^2+ a_i x + b_i$, for $i=1, \dots , n$  and $a_i, b_i, a \in \R$.  Moreover, the determinant of the Bezoutian  $M_n (f) $ is given by
\[   
\D_f = \frac 1 {\mbox{led} (f)} \, \det M_n (f ) =  \displaystyle \prod_{i=1}^n \D_{f_i} \, \cdot \, \prod_{i, j, i\neq j}^n \left( \Res (f_i, f_j, x)   \right)^2
\]
where $\mbox{led} (f)$ is the leading coefficient of $f(x)$. 
     
     ii) the index of inertia of Bezoutian $M(f)$  is 0
    
     iii) if $\D_f\neq 0$ then the equivalence class of $M (f)$  in the Witt ring $W(R)$ is 0.  
     
\end{lem}

\proof
The equivalence between  i), ii), and iii) can be found in \cite{fuh}.

\endproof

It is not clear when such polynomials are  irreducible over $\Q$?  If that's the case, what is the Galois group $\Gal (f)$?
Clearly the group generated by the involution $(1,2)(3,4)\cdots (2n+1, 2n)$ is embedded in $\Gal (f)$.   Is $\Gal (f)$ larger in general?

\section{On the number of real roots of polynomials}\label{Bezoutians}


For any  degree $n\geq 2$ polynomial $f(x) \in \R[x]$ and any  symmetric matrix $M:= M_n(f) $ with real entries, let $N_{f}$ be the \textbf{number of distinct real roots} of $f$ and $\sigma(M)$ be the index of inertia of $M$, respectively. The next result plays a fundamental role throughout this section (\cite[Theorem 9.2]{fuh}).

\begin{prop}\label{prop5.2}
For any real polynomial $f \in \R[x]$, the number $N_f$ of its distinct real roots   is  the index of inertia of the Bezoutian matrix $M_n (f)$. In other words, 
$N_f=\sigma   \left(M_n(f) \right)$.
\end{prop}

Let us cite one more result which says that the roots of a polynomial depend continuously on its coefficients (\cite[Theorem 1.4]{mar}, \cite[Theorem 1.3.1]{rs}).

\begin{prop}\label{prop5.3}
Let $f(x)=\sum_{l=0}^n a_lx^l \in \C[x]$ be a   polynomial   with distinct roots $\alpha_1$, $\cdots$, $\alpha_k$ of multiplicities $m_1$, $\cdots$, $m_k$ respectively.
Then, for any given a positive $\varepsilon<\min_{1\leq i<j \leq k} |\alpha_i-\alpha_j|/2$, there exists a real number $\delta>0$ such that any monic polynomial $g(x)=\sum_{l=0}^nb_lx^l \in \C[x]$ whose coefficients satisfy $|b_l-a_l|< \delta$, for $l=1, \cdots, n-1$, has exactly $m_j$ roots in the disk 
\[
\mathcal{D}(\alpha_j; \varepsilon)=\{ z \in \C \mid |z-\alpha_j|<\epsilon \} \ (j=1,\cdots, k).
\]
\end{prop}

Let $n$, $s$  be positive integers such that $n>s$ and let 
\begin{equation}
\begin{split}
g(t_0,\cdots, t_s ; x) & = \sum_{k=0}^{s}t_{s-k}x^{s-k}, \\
 f^{(n)}(t_0,\cdots, t_s, t ; x) & =x^n+t \cdot g(t_0,\cdots,t_s ; x)\\
\end{split}
\end{equation}
be polynomials in $x$ over $E_1=\R(t_0, \cdots, t_s)$, $E_2=\R(t_0, \cdots, t_s, t)$, respectively. 
Here, $E_1$ (resp., $E_2$) is a rational function field with $s+1$ $(resp., (s+2))$ variables $t_0, \cdots, t_s$ $(resp., (t_0, \cdots, t_s, t))$. 
To ease notation, let us put
\begin{align*}
g(x)=g(t_0,\cdots, t_s ; x), \ f(t ; x)=f^{(n)}(t_0,\cdots, t_s, t ; x)
\end{align*}
and for any real vector ${\bm v}=(v_0, \cdots, v_s) \in \R^{s+1}$, we put
\begin{equation}
\begin{aligned}
g_{{\bm v}}(x)=g(v_0,\cdots, v_s ; x), \; \;  f_{\bm v}(t ; x)=f^{(n)}(v_0,\cdots, v_s, t ; x).
\end{aligned} 
\end{equation}
By using Proposition \ref{prop5.2}, we can prove the next theorem (\cite[Main Theorem 1.3]{otake}).

\begin{thm}\label{thm5.1}
Let ${\bm r}=(r_0, \cdots, r_s) \in \R^{s+1}$ be a vector such that $N_{g_{\bm r}}=s$. Let us consider $f_{\bm r}(t ; x)=f^{(n)}(r_0, \cdots, r_s, t ; x)$ as a polynomial over $\R(t)$ in $x$ and put 
\[ 
P_{\bm r}(t)=\det M_n(f_{\bm r}(t ; x))=\det M_n(f_{\bm r}(t ; x), f_{\bm r}^{\prime}(t ; x)),
\] 
where $f_{\bm r}^{\prime}(t ; x)$ is a derivative of $f_{\bm r}(t ; x)$ with respect to $x$. Then, for any real number $\xi>\alpha_{\bm r}=\max\{ \alpha \in \R \mid P_{\bm r}(\alpha)=0 \}$, we have
\[
N_{f_{\bm r}(\xi ; x)}
=
\begin{cases}
s+1 & \text{if $n-s$ $:$ odd} \\
s & \text{if $n-s$ $:$ even, $r_s>0$} \\
s+2 & \text{if $n-s$ $:$ even, $r_s<0$}.
\end{cases}
\]
\end{thm}
By this theorem and a theorem of Oz Ben-Shimol \cite[Theorem 2.6]{shimol}, we can obtain an algorithm to construct prime degree $p$ polynomials with given number of real roots, and whose Galois groups are isomorphic to the symmetric group $S_p$ or the alternating group $A_p$ (\cite[Corollary 1.6]{otake}).

In this section, we extend this theorem as follows;
\begin{thm} \label{thm5.2}
Let ${\bm r}=(r_0, \cdots, r_s) \in \R^{s+1}$ be a vector such that $g_{\bm r}(x)$ is a degree $s$ separable polynomial satisfying $N_{g_{\bm r}(x)}=\gamma$ $(0 \leq \gamma \leq s)$. Let us consider $f_{\bm r}(t ; x)=f^{(n)}(r_0, \cdots, r_s, t ; x)$ as a polynomial over $\R(t)$ in $x$ and put 
\[ 
P_{\bm r}(t)=\det M_n(f_{\bm r}(t ; x))=\det M_n(f_{\bm r}(t ; x), f_{\bm r}^{\prime}(t ; x)),
\] 
where $f_{\bm r}^{\prime}(t ; x)$ is a derivative of $f_{\bm r}(t ; x)$ with respect to $x$. 
Then, for any real number $\xi>\alpha_{\bm r}=\max\{ \alpha \in \R \mid P_{\bm r}(\alpha)=0 \}$, we have 
\begin{equation}\label{eq-thm2}
\begin{aligned}
N_{f_{\bm r}(\xi ; x)}
=
\begin{cases}
\gamma+1 & \text{if $n-s$ $:$ odd} \\
\gamma & \text{if $n-s$ $:$ even, $r_s>0$} \\
\gamma+2 & \text{if $n-s$ $:$ even, $r_s<0$}.
\end{cases}
\end{aligned}
\end{equation}
\end{thm}

The above theorem can be restated as follows:

\begin{cor}
Let $f \in \R( t) [x]$ be given by
\[ f(t, x) = x^n + t \cdot \sum_{k=0}^s t_{s-k} x^{s-k} \]
and $\beta_1 <  \dots < \beta_m$ the distinct real roots of the degree $s$ polynomial  
\[ P(t) := \frac 1 {t^{n-1}} \, \D_{(f, x)} (t).\]
 For any $\xi > | \beta_m |$, the number of real roots of $f(\xi, x)$ is 
\[
\begin{aligned}
N_{f (\xi , x)}
=
\begin{cases}
\gamma+1 & \text{if $n-s$ $:$ odd} \\
\gamma & \text{if $n-s$ $:$ even, $t_s>0$} \\
\gamma+2 & \text{if $n-s$ $:$ even, $t_s<0$}.
\end{cases}
\end{aligned}
\]
where $\gamma$ is the number or real roots of $g(x) = \frac {f(x) - x^n } t \in \R[x]$.
\end{cor}

The rest of the section is concerned with proving Thm.~\ref{thm5.2}. 
\subsection{The Bezoutian of $f(t; x)$}

First, let us put 
\begin{align*}
&A(t_0, \cdots, t_s, t)=(a_{ij}(t_0, \cdots, t_s, t))_{1 \leq i,j \leq n}=M_n(f(t ; x)) \in \mathrm{Sym}_n(E_2), \\
&B(t_0, \cdots, t_s)=(b_{ij}(t_0, \cdots, t_s))_{1 \leq i,j \leq s}=M_s(g(x)) \in \mathrm{Sym}_s(E_1). 
\end{align*}
For ease of notation, we also write
\begin{align*}
A(t_0, \cdots, t_s, t)=A(t)=(a_{ij}(t))_{1 \leq i,j \leq n}, \ B(t_0, \cdots, t_s)=B=(b_{ij})_{1 \leq i,j \leq s}
\end{align*}
and we put $B(t)=(b_{ij}(t))_{1 \leq i,j \leq s}=t^2B$. 
Then, by Proposition \ref{prop5.1}, we have
\begin{align*}
A(t)&=M_n(x^n+tg(x), nx^{n-1}+tg^{\prime}(x)) \\ 
&=nM_n(x^n, x^{n-1})-ntM_n(x^{n-1}, g(x))+tM_n(x^n, g^{\prime}(x))+t^2M_n(g(x), g^{\prime}(x)) \\
&=nM_n(x^n, x^{n-1})-nt\sum_{k=0}^{s}t_{s-k}M_n(x^{n-1}, x^{s-k}) \\
&\hspace{30.3mm}+t\sum_{k=0}^{s-1}(s-k)t_{s-k}M_n(x^n, x^{s-k-1})+t^2M_n(g(x), g^{\prime}(x)).
\end{align*}
\normalsize
\begin{lem} \label{lem5.1}
Let $\lambda ,\mu ,\nu$ be integers such that $\lambda \geq \mu > \nu \geq 0$. Then
$M_{\lambda }(x^{\mu }, x^{\nu }) = (m_{ij} )_{1\leq i,j\leq \lambda }$,
where
\begin{equation*}
m_{ij}=\begin{cases}
               1 & \text{$i+j=2\lambda-(\mu+\nu)+1$ \ $(\lambda-\mu+1\leq i,j\leq \lambda-\nu )$}, \\
               0 & \text{otherwise}.
             \end{cases}
\end{equation*}
\end{lem}
\begin{proof}
By definition, we have
\begin{align*}
B_{\lambda}(x^{\mu}, x^{\nu})&=\dfrac{x^{\mu}y^{\nu}-x^{\nu}y^{\mu}}{x-y} \\
&=\sum_{k=1}^{\mu-\nu} x^{\mu-k}y^{\nu+k-1}
=\sum_{k=1}^{\mu-\nu}x^{\lambda-(\lambda-\mu+k)}y^{\lambda-(\lambda-\nu-k+1)},
\end{align*}
which implies
\begin{align*}
m_{ij}
&=
\begin{cases}
1 & \text{$(i,j)=(\lambda-\mu+k, \lambda-\nu-k+1)$ \ $(1 \leq k \leq \mu-\nu)$} \\
0 & \text{otherwise}
\end{cases} \\
&=
\begin{cases}
1 & \text{$i+j=2\lambda-(\mu+\nu)+1$ \ $(\lambda-\mu+1\leq i,j\leq \lambda-\nu )$}, \\
0 & \text{otherwise}.
\end{cases}
\end{align*}
This completes the proof.
\end{proof}
Here, let us divide $A(t)$ into two parts $\hat{A}(t)$ and $\tilde{A}(t)$, where
\begin{align*}
\hat{A}(t)&=(\hat{a}_{ij}(t))_{1\leq i, j \leq n}=nM_n(x^n, x^{n-1})-nt\sum_{k=0}^{s}t_{s-k}M_n(x^{n-1}, x^{s-k}) \\
& \hspace{56.5mm} +t\sum_{k=0}^{s-1}(s-k)t_{s-k}M_n(x^n, x^{s-k-1}), \\
\tilde{A}(t)&=(\tilde{a}_{ij}(t))_{1\leq i, j \leq n}=t^2M_n(g(x), g^{\prime}(x))
\end{align*}
and put $l_k=n-s+k+2 \ (=2n-(n+s-k-1)+1)$. Then, by lemma \ref{lem5.1}, we have 
\begin{align*}
\begin{cases}
\hat{a}_{11}(t)=n \\
\hat{a}_{1,l_k-1}(t)=\hat{a}_{l_k-1,1}(t)=(s-k)t_{s-k}t \ (0 \leq k \leq s-1).
\end{cases}
\end{align*}
Moreover, when $i+j=l_k$, we have
\small
\begin{align} \label{eq5.1}
\hat{a}_{ij}(t)&=-ntt_{s-k}+t(s-k)t_{s-k}=-(l_k-2)t_{s-k}t \hspace{2mm} (2 \leq i, j \leq l_k-2, \hspace{0.5mm} 0 \leq k\leq s).
\end{align}
\normalsize
\begin{rem}
Note that, if $s=n-1$, we have
\begin{align*}
-nt\sum_{k=0}^{s}t_{s-k}M_n(x^{n-1}, x^{s-k})=-nt\sum_{k=1}^{s}t_{s-k}M_n(x^{n-1}, x^{s-k}),
\end{align*}
Thus, when $i+j=l_k$, equation \eqref{eq5.1} should be modified by
\begin{align*} 
\hat{a}_{ij}(t)=-ntt_{s-k}+t(s-k)t_{s-k}=-(l_k-2)t_{s-k}t \hspace{3mm} (2 \leq i, j \leq l_k-2, \hspace{0.5mm} 1 \leq k\leq s). 
\end{align*}
We avoid this minor defect by considering that there is no entries satisfying $2 \leq i,j \leq l_0-2$ when $s=n-1$ since $l_0-2=n-s=1$.
\end{rem}
\begin{prop} \label{prop5.4}
Put $l_k=n-s+k+2$. Then 
\begin{align*}
&\hat{a}_{ij}(t)
=
\begin{cases}
n & \text{$(i,j)=(1, 1)$} \\
(s-k)t_{s-k}t & \text{$(i,j)=(1, l_k-1)$ or $(l_k-1, 1)$ \ $(0 \leq k \leq s-1)$} \\
-(l_k-2)t_{s-k}t & \text{$i+j=l_k$, $2 \leq i,j \leq l_k-2$, $(0 \leq k \leq s)$} \\
0 & \text{otherwise}.
\end{cases} \\
&\tilde{a}_{ij}(t)
=
\begin{cases}
b_{i-(n-s), \hspace{0.3mm} j-(n-s)}t^2 & \text{$n-s+1 \leq i, j \leq n$} \\
0 & \text{otherwise}.
\end{cases}
\end{align*}
\end{prop}
\begin{proof}
The statement for $\hat{a}_{ij}(t)$ has just been proved. For $\tilde{a}_{ij}(t)$, it is enough to see that we can denote
\begin{align*}
&M_s(g(x))=\sum_{\ell=0}^s\sum_{m=1}^s mt_{\ell}t_mM_s(x^{\ell}, x^{m-1}), \\  
&M_n(g(x))=\sum_{\ell=0}^s\sum_{m=1}^s mt_{\ell}t_mM_n(x^{\ell}, x^{m-1}),
\end{align*}
that is, we can obtain $M_n(g(x))$ from $M_s(g(x))$ by just replacing $s$ with $n$ for all $M_s(x^{\ell}, x^m)$, which, by Lemma \ref{lem5.1}, means that 
$s \times s$ matrix  $M_s(g(x))$ occupies the part $\{ b_{ij}^{\dagger} \mid n-s+1 \leq i,j \leq n \}$ of the matrix $M_n(g(x))=(b_{ij}^{\dagger})_{1 \leq i,j \leq n}$.
\end{proof}
By Proposition \ref{prop5.4}, we can express the matrix $A(t)$ as follows;
\footnotesize{
\begin{align} \label{eq5.11}
A(t)
=
\scalebox{0.83}[1]
{$\left[
\begin{array}{cccc|cccc}
\hspace{-1.5mm} n & \hspace{-4mm} 0 & \hspace{-3mm} \dots & \hspace{-2mm} 0 & \hspace{-1mm} st_st & \hspace{-3mm} (s-1)t_{s-1}t & \hspace{-2.5mm} \dots & \hspace{-1.5mm} t_1t \\
\hspace{-1.5mm}0 &&& \hspace{-2mm} -(n-s)t_st \hspace{-0.5mm} & \hspace{-0.5mm} -(n-s+1)t_{s-1}t & \hspace{-3mm} \dots & \hspace{-2.5mm} -(n-1)t_1t & \hspace{-1.5mm} -nt_0t \\
\hspace{-1.5mm}\vdots && \hspace{-3mm} \rotatebox[origin=c]{295}{\vdots} & \hspace{-2mm} \rotatebox[origin=c]{295}{\vdots} && \hspace{-3mm} \rotatebox[origin=c]{295}{\vdots} & \hspace{-2.5mm} \rotatebox[origin=c]{295}{\vdots} & \hspace{-1.5mm} 0  \\
\hspace{-1.5mm}0 & \hspace{-4mm} -(n-s)t_st & \hspace{-3mm} \rotatebox[origin=c]{295}{\vdots} &&& \hspace{-3mm} \rotatebox[origin=c]{295}{\vdots} & \hspace{-2.5mm} 0 & \hspace{-1.5mm} 0 \\ \cline{1-8}
\hspace{-1.5mm}st_st & \hspace{-3mm} -(n-s+1)t_{s-1}t &&&&&& \\
\hspace{-1.5mm}(s-1)t_{s-1}t & \hspace{-4mm} \vdots & \hspace{-3mm} \rotatebox[origin=c]{295}{\vdots} & \hspace{-2mm} \rotatebox[origin=c]{295}{\vdots} & \multicolumn{4}{c}{\raisebox{-10pt}[0pt][0pt]{$C(t)$}} \\
\vdots & \hspace{-4mm} -(n-1)t_1t & \hspace{-3mm} \rotatebox[origin=c]{295}{\vdots} & \hspace{-2mm} 0 &&&& \\  
\hspace{-1.5mm}t_1t & \hspace{-4mm} -nt_0t & \hspace{-3mm} 0 & \hspace{-2mm} 0 &&&& \\
\end{array}
\hspace{-0.5mm}\right]
$}.  
\end{align}
}
\normalsize
Here, $C(t)=(c_{ij}(t))_{1\leq i,j \leq s}=C(t_0, \cdots, t_s,t)=(c_{ij}(t_0, \cdots, t_s,t))_{1 \leq i,j \leq s}$ is an $s \times s$ symmetric matrix whose entries are of the form 
\begin{align*} 
c_{ij}(t_0, \cdots, t_s,t)&=b_{ij}t^2+\lambda_{ij}t \\
&=b_{ij}(t_0, \cdots, t_s)t^2+\lambda_{ij}(t_0, \cdots, t_s)t \hspace{3mm} (\lambda_{ij}=\lambda_{ij}(t_0, \cdots, t_s) \in E_1). 
\end{align*}
Next, let 
$A(t)_1=(a_{ij}(t)_1)_{1\leq i,j \leq n}=A(t_0, \cdots, t_s, t)_1=(a_{ij}(t_0, \cdots, t_s,t)_1)_{1\leq i,j \leq n}$
be the $n \times n$ symmetric matrix obtained from $A(t)$ by multiplying the first row and the first column by $1/\sqrt{n}$ and then sweeping out the entries of the first row and the first column by the $(1,1)$ entry $1$. 
Here, let $Q_{m}(k;c)=(q _{ij})_{1\leq i,j\leq m}$ and $R_{m}(k,l;c)=(r _{ij})_{1\leq i,j\leq m}$ be $m\times m$ elementary matrices such that
\begin{center}
\noindent $Q_m(k;c)$=
\scalebox{0.85}[1]
{$\left[
\begin{array}{ccccccc} 
1 & & & & & & \\ 
& \ddots & & & & & \\ 
& & 1 & & & & \\ 
& & & c & & & \\ 
& & & & 1 & & \\ 
& & & & & \ddots & \\ 
& & & & & & 1
\end{array}
\right]
$},
$R_m(k,l;c)$=
\scalebox{0.85}[1]
{$\left[
\begin{array}{cccccccc} 
1 & & & & & & & \\ 
& \ddots & & & & & & \\ 
& & 1 & & & c & \\ 
& & & \ddots & & & & \\ 
& & & & & 1 & & \\ 
& & & & & & \ddots & \\ 
& & & & & & & 1
\end{array}
\right],
$}
\end{center}
where $q _{kk}=c$ and $r _{kl}=c$. Moreover, for any $m \times m$ matrices $M_1$, $M_2$, $\cdots$, $M_l$, put
$\prod_{k=1}^{l} M_k=M_1M_2\cdots M_l$. Then, we have $A(t)_1={}^tS(t)_1A(t)S(t)_1$, where 
\begin{align*}
S(t)_1=\displaystyle Q_{n}(1;1/\sqrt{n})\prod_{k=0}^{s-1}R_n(1,l_k-1;-a_{1, l_k-1}(t)/\sqrt{n}). 
\end{align*}
The matrix $A(t)_1$ can be expressed as follows;
\footnotesize{
\begin{align}\label{eq5.2}
A(t)_1=
\scalebox{0.99}[1]
{$\left[
\begin{array}{cccc|cccc}
\hspace{-1.5mm} 1 & \hspace{-4mm} 0 & \hspace{-3mm} \dots & \hspace{-2mm} 0 & \hspace{-1mm} 0 & \hspace{-3mm} 0 & \hspace{-2.5mm} \dots & \hspace{-1.5mm} 0 \\
\hspace{-1.5mm}0 & \hspace{-4mm} 0 & \hspace{-3mm} \dots & \hspace{-2mm} -(n-s)t_st \hspace{-0.5mm} & \hspace{-0.5mm} -(n-s+1)t_{s-1}t & \hspace{-3mm} \dots & \hspace{-2.5mm} -(n-1)t_1t & \hspace{-1.5mm} -nt_0t \\
\hspace{-1.5mm}\vdots & \hspace{-4mm} \vdots & \hspace{-3mm} \rotatebox[origin=c]{295}{\vdots} & \hspace{-2mm} \rotatebox[origin=c]{295}{\vdots} && \hspace{-3mm} \rotatebox[origin=c]{295}{\vdots} & \hspace{-2.5mm} \rotatebox[origin=c]{295}{\vdots} & \hspace{-1.5mm} 0  \\
\hspace{-1.5mm}0 & \hspace{-4mm} -(n-s)t_st & \hspace{-3mm} \rotatebox[origin=c]{295}{\vdots} &&& \hspace{-3mm} \rotatebox[origin=c]{295}{\vdots} & \hspace{-2.5mm} 0 & \hspace{-1.5mm} 0 \\ \cline{1-8}
\hspace{-1.5mm}0 & \hspace{-3mm} -(n-s+1)t_{s-1}t &&&&&& \\
\hspace{-1.5mm}0 & \hspace{-4mm} \vdots & \hspace{-3mm} \rotatebox[origin=c]{295}{\vdots} & \hspace{-2mm} \rotatebox[origin=c]{295}{\vdots} & \multicolumn{4}{c}{\raisebox{-10pt}[0pt][0pt]{$C(t)_1$}} \\
\vdots & \hspace{-4mm} -(n-1)t_1t & \hspace{-3mm} \rotatebox[origin=c]{295}{\vdots} & \hspace{-2mm} 0 &&&& \\  
\hspace{-1.5mm}0 & \hspace{-4mm} -nt_0t & \hspace{-3mm} 0 & \hspace{-2mm} 0 &&&& \\
\end{array}
\hspace{-0.5mm}\right]
$}.
\end{align} 
\normalsize
Here, 
$C(t)_1=(c_{ij}(t)_1)_{1\leq i,j \leq s}=C(t_0, \cdots, t_s,t)_1=(c_{ij}(t_0, \cdots, t_s,t)_1)_{1\leq i,j \leq s}$
is an $s \times s$ symmetric matrix whose entries are of the form 
\begin{align*}
c_{ij}(t_0, \cdots, t_s,t)_1=\bar{b}_{ij}(t_0, \cdots, t_s)t^2+\lambda_{ij}(t_0, \cdots, t_s)t \hspace{3mm} (\bar{b}_{ij}(t_0, \cdots, t_s) \in E_1), 
\end{align*}
where
\begin{align} \label{eq5.22}
\bar{b}_{ij}(t_0, \cdots, t_s)
&=b_{ij}(t_0, \cdots, t_s)-\dfrac{(s-i+1)(s-j+1)}{n}t_{s-i+1}t_{s-j+1} 
\end{align}
for any $i, j$ ($1 \leq i,j \leq s$). We put $\bar{b}_{ij}(t_0, \cdots, t_s)=\bar{b}_{ij}$ and $\bar{B}=(\bar{b}_{ij})_{1 \leq i,j \leq s}$.

\subsection{Some results for the Bezoutian of $f_{\text{\boldmath $r$}}(t ; x)$}

Let ${\bm r}=(r_0, \cdots, r_s) \in \R^{s+1}$ be a vector as in Theorem \ref{thm5.2}. We put
\begin{align*}
&A_{\bm r}(t)=(a_{ij}^{({\bm r})}(t))_{1\leq i, j \leq n}=A(r_0, \cdots, r_s, t) \in \mathrm{Sym}_n(\R(t)), \\
&B_{\bm r}=(b_{ij}^{({\bm r})})_{1\leq i,j \leq s}=B(r_0, \cdots, r_s) \in \mathrm{Sym}_s(\R)
\end{align*}
and $B_{\bm r}(t)=t^2B_{\bm r}$. 
Let us also put $A_{\bm r}(t)_1=A(r_0, \cdots, r_s, t)_1$. By equation \eqref{eq5.2}, the matrix $A_{\bm r}(t)_1$ can be expressed as follows;
\footnotesize{
\begin{align*}
A_{\bm r}(t)_1=
\scalebox{0.98}[1]
{$\left[
\begin{array}{cccc|cccc}
\hspace{-1.5mm} 1 & \hspace{-4mm} 0 & \hspace{-3mm} \dots & \hspace{-2mm} 0 & \hspace{-1mm} 0 & \hspace{-3mm} 0 & \hspace{-2.5mm} \dots & \hspace{-1.5mm} 0 \\
\hspace{-1.5mm}0 & \hspace{-4mm} 0 & \hspace{-3mm} \dots & \hspace{-2mm} -(n-s)r_st \hspace{-0.5mm} & \hspace{-0.5mm} -(n-s+1)r_{s-1}t & \hspace{-3mm} \dots & \hspace{-2.5mm} -(n-1)r_1t & \hspace{-1.5mm} -nr_0t \\
\hspace{-1.5mm}\vdots & \hspace{-4mm} \vdots & \hspace{-3mm} \rotatebox[origin=c]{295}{\vdots} & \hspace{-2mm} \rotatebox[origin=c]{295}{\vdots} && \hspace{-3mm} \rotatebox[origin=c]{295}{\vdots} & \hspace{-2.5mm} \rotatebox[origin=c]{295}{\vdots} & \hspace{-1.5mm} 0  \\
\hspace{-1.5mm}0 & \hspace{-4mm} -(n-s)r_st & \hspace{-3mm} \rotatebox[origin=c]{295}{\vdots} &&& \hspace{-3mm} \rotatebox[origin=c]{295}{\vdots} & \hspace{-2.5mm} 0 & \hspace{-1.5mm} 0 \\ \cline{1-8}
\hspace{-1.5mm}0 & \hspace{-3mm} -(n-s+1)r_{s-1}t &&&&&& \\
\hspace{-1.5mm}0 & \hspace{-4mm} \vdots & \hspace{-3mm} \rotatebox[origin=c]{295}{\vdots} & \hspace{-2mm} \rotatebox[origin=c]{295}{\vdots} & \multicolumn{4}{c}{\raisebox{-10pt}[0pt][0pt]{$C_{\bm r}(t)_1$}} \\
\vdots & \hspace{-4mm} -(n-1)r_1t & \hspace{-3mm} \rotatebox[origin=c]{295}{\vdots} & \hspace{-2mm} 0 &&&& \\  
\hspace{-1.5mm}0 & \hspace{-4mm} -nr_0t & \hspace{-3mm} 0 & \hspace{-2mm} 0 &&&& \\
\end{array}
\hspace{-0.5mm}\right]
$}.
\end{align*} 
\normalsize
Here, 
$C_{\bm r}(t)_1=(c_{ij}^{(\bm r)}(t)_1)_{1\leq i,j \leq s}=C(r_0, \cdots, r_s,t)_1$ and
\begin{align*}
c_{ij}^{(\bm r)}(t)_1=\bar{b}_{ij}(r_0, \cdots, r_s)t^2+\lambda_{ij}(r_0, \cdots, r_s)t \hspace{3mm} (\bar{b}_{ij}(r_0, \cdots, r_s), \lambda_{ij}(r_0, \cdots, r_s) \in \R). 
\end{align*}
Note that, by equation \eqref{eq5.22}, we have
\begin{align*} 
\bar{b}_{ij}(r_0, \cdots, r_s)
&=b_{ij}^{({\bm r})}-\dfrac{(s-i+1)(s-j+1)}{n}r_{s-i+1}r_{s-j+1} \hspace{3mm} (1 \leq i,j \leq s).
\end{align*}
To ease notation, we put $\bar{b}_{ij}(r_0, \cdots, r_s)=\bar{b}_{ij}^{({\bm r})}$ and $\bar{B}_{\bm r}=(\bar{b}_{ij}^{(\bm r)})_{1 \leq i,j \leq s}$.

In particular, since
\begin{align*}
M_s(g_{\bm r})
&=M_s\left(r_sx^s, \sum_{k=0}^{s-1}(s-k)r_{s-k}x^{s-k-1}\right)+M_s\left(\sum_{k=1}^{s}r_{s-k}x^{s-k}, g_{\bm r}^{\prime}\right) \\
&=\sum_{k=0}^{s-1}(s-k)r_sr_{s-k}M_s(x^s,x^{s-k-1})+M_s\left(\sum_{k=1}^{s}r_{s-k}x^{s-k}, g_{\bm r}^{\prime}\right),
\end{align*}
we have
\begin{align} \label{eq5.3}
b_{1, k+1}^{(\bm r)}=b_{k+1, 1}^{(\bm r)}=(s-k)r_sr_{s-k} \ (0 \leq k \leq s-1)
\end{align}
by Lemma \ref{lem5.1} and hence
\begin{align} \label{eq5.4}
\bar{b}_{1j}^{({\bm r})}&=(s-j+1)r_sr_{s-j+1}-\dfrac{s(s-j+1)}{n}r_{s}r_{s-j+1} \\ \notag 
&=(s-j+1)\left(1-\frac{s}{n}\right)r_sr_{s-j+1} \ (1\leq j \leq s).
\end{align}
\begin{lem}\label{lem5.2}
Put $\bar{B}_{\bm r}(t)=t^2\bar{B}_{\bm r}$. Then, $B_{\bm r}(\xi)$ and $\bar{B}_{\bm r}(\xi)$ are equivalent over $\R$ for any real number $\xi$ and we have $\sigma(\bar{B}_{\bm r}(\xi))=N_{g_{\bm r}}$ for any non-zero real number $\xi$.
\end{lem}
\begin{proof}
Let us denote by $B_{\bm r}^{*}=(b_{ij}^{({\bm r}, *)})_{1\leq i,j \leq s}$ (${\bar{B}_{\bm r}}^{*}=(\bar{b}_{ij}^{({\bm r}, *)})_{1 \leq i,j \leq s}$) the matrix obtained from $B_{\bm r}$ ($\bar{B}_{\bm r}$) by multiplying the first row and the first column by $1\Bigl/\pm\sqrt{b_{11}^{(\bm r)}}$ $\left(1\Bigl/\pm\sqrt{\bar{b}_{11}^{(\bm r)}}\right)$ 
(the sign before $\sqrt{b_{11}^{(\bm r)}}$ $\left( \sqrt{\bar{b}_{11}^{(\bm r)}} \right)$ are the same as the sign of $r_s$; see the definition of $d$ $\left( \bar{d} \hspace{0.7mm} \right)$ below) 
and then sweeping out the entries of the first row and the first column by the $(1,1)$ entry $1$. Since $b_{11}=sr_s^2$ $(>0)$ and $\bar{b}_{11}=s(1-s/n)r_s^2$ $(>0)$ by \eqref{eq5.3} and \eqref{eq5.4}, we have
\begin{align} \label{eq5.41}
B_{\bm r}^{*}={}^tTB_{\bm r}T, \ \bar{B}_{\bm r}^{*}={}^t\bar{T}\bar{B}_{\bm r}\bar{T},
\end{align}
where
\begin{align*}
T&=Q_s(1;1/d)\prod_{k=2}^{s}R_s(1,k;-b_{1k}^{(\bm r)}/d) \ (d=\sqrt{s}\cdot r_s), \\
\bar{T}&=Q_s(1;1/\bar{d})\prod_{k=2}^{s}R_s(1,k;-\bar{b}_{1k}^{(\bm r)}/\bar{d}) \ (\bar{d}=\sqrt{s(1-s/n)} \cdot r_s).
\end{align*}
Note that in \cite[Lemma 3.3]{otake}, we have proved $b_{ij}^{(\bm r, *)}=\bar{b}_{ij}^{(\bm r, *)}$ $(1 \leq i,j \leq s)$ and hence $t^2B_{\bm r}^{*}=t^2\bar{B}_{\bm r}^{*}$, which, by \eqref{eq5.41}, implies that symmetric matrices $B_{\bm r}(\xi)$ and $\bar{B}_{\bm r}(\xi)$ are equivalent over $\R$ for any real number $\xi$. Then, since $N_{g_{\bm r}}=\sigma(B_{\bm r})=\sigma(B_{\bm r}(\xi))$ for any $\xi \in \R\setminus\{0\}$, the latter half of the statement have also been proved.
\end{proof}

\subsection{Nonvanishingness of some coefficients}
In this subsection, we prove the next lemma.
\begin{lem}\label{lem5.3}
Let 
\begin{align}\label{eq5.54}
\varPhi(x)=\varPhi(t_0, \cdots, t_s ; x)=\sum_{k=0}^{s}h_{s-k}(t_0, \cdots, t_s)x^{s-k} \in E_1[x] 
\end{align}
be the characteristic polynomial of $\bar{B}$. Then, $h_{s-k}(t_0, \cdots, t_s)$ is a non-zero polynomial in $E_1$ for any $k$ $(1 \leq k \leq s)$. 
\end{lem}
\begin{proof}
Lemma \ref{lem5.3} is clear for $s=1$, since we have
\begin{align*}
B=M_1(t_1x+t_0)
=
\left[
\begin{array}{c}
t_1^2
\end{array}
\right]
\end{align*} 
and hence, by equation \eqref{eq5.22},
\begin{align*}
\bar{B}=
\left[
\begin{array}{c}
t_1^2-\dfrac{1}{n}t_1^2
\end{array}
\right]
=
\left[
\begin{array}{c}
\dfrac{n-1}{n}t_1^2
\end{array}
\right].
\end{align*}
Next, suppose $s \geq 2$. Then, by equation \eqref{eq5.22} and the definition of the Bezoutian, we have $h_{s-k}(t_0, \cdots, t_s) \in \R[t_0, \cdots, t_s]$ for any $k$ $(1 \leq k \leq s)$. Thus, we have only to prove that $h_{s-k}(t_0, \cdots, t_s) \neq 0$ for any $k$ $(1 \leq k \leq s)$, which is clear from the next Lemma \ref{lem5.4}.
\end{proof}
\begin{lem}\label{lem5.4}
Suppose $s\geq2$ and put $u_0=u_s=1$, $u_1=t_1$ and $u_k=0$ $(2 \leq k \leq s-1)$. Then, $h_{s-k}(u_0, \cdots, u_s)$ is a non-constant polynomial in $\R(t_1)$ for any $k$ $(1 \leq k \leq s)$, i.e., $h_{s-k}(u_0, \cdots, u_s) \in \R[t_1] \setminus \R$ $(1 \leq k \leq s)$.  
\end{lem}
\noindent To prove lemma \ref{lem5.4}, let us put ${\bm u}=(u_0, \cdots, u_s)$ and
\begin{align*}
&g_{{\bm u}}(x)=g(u_0, \cdots, u_s; x)=x^s+t_1x+1 \in \R(t_1)[x], \\
&f_{\bm u}(t ; x)=x^n+tg_{{\bm u}}(x) \in \R(t_1, t)[x] \ (n>s), \\
&A_{{\bm u}}(t)=(a_{ij}^{({\bm u})}(t))_{1 \leq i,j \leq n}=A(u_0, \cdots, u_s,t) \in \mathrm{Sym}_n(\R(t_1, t)), \\
&B_{{\bm u}}=(b_{ij}^{({\bm u})})_{1 \leq i,j \leq s}=B(t_0, \cdots, u_s) \in \mathrm{Sym}_s(\R(t_1)), \ B_{{\bm u}}(t)=t^2B_{{\bm u}}.
\end{align*}
Then, by equation \eqref{eq5.11}, we have
\footnotesize{
\begin{align*} 
A_{{\bm u}}(t)
=
\scalebox{1}[1]
{$\left[
\begin{array}{cccc|cccc}
n & 0 & \dots & 0 & st & 0 & \dots & t_1t \\
0 &&& -(n-s)t & 0 & \dots & -(n-1)t_1t & -nt \\
\vdots && \rotatebox[origin=c]{295}{\vdots} & \rotatebox[origin=c]{295}{\vdots} && \rotatebox[origin=c]{295}{\vdots} & \rotatebox[origin=c]{295}{\vdots} & 0  \\
0 & -(n-s)t & \rotatebox[origin=c]{295}{\vdots} &&& \rotatebox[origin=c]{295}{\vdots} & 0 & 0 \\ \cline{1-8}
st & 0 &&&&&& \\
0 & \vdots & \rotatebox[origin=c]{295}{\vdots} & \rotatebox[origin=c]{295}{\vdots} & \multicolumn{4}{c}{\raisebox{-10pt}[0pt][0pt]{$C_{{\bm u}}(t)$}} \\
\vdots & -(n-1)t_1t & \rotatebox[origin=c]{295}{\vdots} & 0 &&&& \\  
t_1t & -nt & 0 & 0 &&&& \\
\end{array}
\right]
$},
\end{align*}
}
\normalsize
where $C_{{\bm u}}(t)=(c_{ij}^{({\bm u})}(t))_{1\leq i,j \leq s}=C(u_0, \cdots, u_s,t)$ and 
\begin{align*}
c_{ij}^{(\bm u)}(t)=b_{ij}(u_0, \cdots, u_s)t^2+\lambda_{ij}(u_0, \cdots, u_s)t \hspace{3mm} (\lambda_{ij}(u_0, \cdots, u_s) \in \R(t_1)). 
\end{align*}
Moreover, by equation \eqref{eq5.2}, we also have 
\footnotesize{
\begin{align*}
A_{\bm u}(t)_1=
\scalebox{0.99}[1]
{$\left[
\begin{array}{cccc|cccc}
1 & 0 & \dots & 0 & 0 & 0 & \dots & 0 \\
0 & 0 & \dots & -(n-s)t & 0 & \dots & -(n-1)t_1t & -nt \\
\vdots & \vdots & \rotatebox[origin=c]{295}{\vdots} & \rotatebox[origin=c]{295}{\vdots} && \rotatebox[origin=c]{295}{\vdots} & \rotatebox[origin=c]{295}{\vdots} & 0  \\
0 & -(n-s)t & \rotatebox[origin=c]{295}{\vdots} &&& \rotatebox[origin=c]{295}{\vdots} & 0 & 0 \\ \cline{1-8}
0 & 0 &&&&&& \\
0 & \vdots & \rotatebox[origin=c]{295}{\vdots} & \rotatebox[origin=c]{295}{\vdots} & \multicolumn{4}{c}{\raisebox{-10pt}[0pt][0pt]{$C_{\bm u}(t)_1$}} \\
\vdots & -(n-1)t_1t & \hspace{-3mm} \rotatebox[origin=c]{295}{\vdots} & 0 &&&& \\  
0 & \hspace{-4mm} -nt & 0 & 0 &&&& \\
\end{array}
\right]
$}.
\end{align*} 
\normalsize
Here, 
$C_{\bm u}(t)_1=(c_{ij}^{(\bm u)}(t)_1)_{1\leq i,j \leq s}=C(u_0, \cdots, u_s,t)_1$ and
\begin{align*}
c_{ij}^{(\bm u)}(t)_1=\bar{b}_{ij}(u_0, \cdots, u_s)t^2+\lambda_{ij}(u_0, \cdots, u_s)t \hspace{3mm} (\bar{b}_{ij}(u_0, \cdots, u_s) \in \R). 
\end{align*}
Note that, by equation \eqref{eq5.22}, we have
\begin{align} \label{eq5.6} 
\bar{b}_{ij}^{({\bm u})}
=
\begin{cases}
b_{11}^{({\bm u})}-(s^2/n) & \text{$(i,j)=(1,1)$} \\[0.8mm]
b_{1s}^{({\bm u})}-(s/n)t_{1} & \text{$(i,j)=(1,s) \ \text{or} \ (s,1)$} \\[0.8mm]
b_{ss}^{({\bm u})}-(1/n)t_1^2 & \text{$(i,j)=(s,s)$} \\[0.8mm] 
b_{ij}^{({\bm u})} & \text{otherwise}.
\end{cases}
\end{align}
Let us put $\bar{B}_{\bm u}=(\bar{b}_{ij}^{({\bm u})})_{1 \leq i,j \leq s}$ and $\bar{B}_{\bm u}(t)=t^2\bar{B}_{\bm u}$.
Then, since 
\begin{align*}
M_s(g_{\bm u})&=M_s(x^s+t_1x+1, sx^{s-1}+t_1) \\
&=sM_s(x^s, x^{s-1})+t_1M_s(x^s, 1)-st_1M_s(x^{s-1}, x)-sM_s(x^{s-1}, 1) \\
&\hspace{74mm} +t_1^2M_s(x, 1)+t_1M_s(1, 1),
\end{align*}
we have \\[1.5mm]
$(a)$ if $s=2$,
\begin{align*}
B_{\bm u}
=
\left[
\begin{array}{cc}
2 & t_1 \\
t_1 & t_1^2-2
\end{array}
\right],
\end{align*}
$(b)$ if $s\geq 3$,
\begin{align*}
b_{ij}^{({\bm u})}
=
\begin{cases}
s & \text{$(i, j)=(1, 1)$} \\
t_1 & \text{$(i, j)=(1, s)$ or $(s, 1)$} \\
(1-s)t_1 & \text{$i+j=s+1$, $2 \leq i,j \leq s-1$} \\
-s & \text{$i+j=s+2$} \\
t_1^2 & \text{$(i, j)=(s, s)$}, \\
0 & \text{otherwise},
\end{cases}
\end{align*}
which, by equation \eqref{eq5.6}, implies \\[1.5mm]
$(a^{\prime})$ if $s=2$,
\begin{align*}
\bar{B}_{\bm u}
=
\left[
\begin{array}{cc}
2(n-2)/n & (n-2)t_1/n \\
(n-2)t_1/n & (n-1)t_1^2/n-2
\end{array}
\right],
\end{align*}
$(b^{\prime})$ if $s \geq 3$,
\begin{align*}  
\bar{b}_{ij}^{({\bm u})}
=
\begin{cases}
s(n-s)/n & \text{$(i,j)=(1,1)$} \\
(n-s)t_1/n & \text{$(i,j)=(1,s) \ \text{or} \ (s,1)$} \\
(1-s)t_1 & \text{$i+j=s+1$, $2 \leq i,j \leq s-1$} \\
-s & \text{$i+j=s+2$} \\
(n-1)t_1^2/n & \text{$(i,j)=(s,s)$}, \\  
0 & \text{otherwise}.
\end{cases}
\end{align*}
Therefore, if $s \geq 3$, the matrix $\bar{B}_{\bm u}=(\bar{b}_{ij}^{({\bm u})})_{1\leq i,j \leq s}$ has the expression of the form
\begin{align*}
\scalebox{0.93}[1]
{$\left[
\begin{array}{ccccccc}
s(n-s)/n & 0 & 0 & 0 & \cdots & 0 & (n-s)t_1/n \\
0 & 0 & 0 & \cdots & 0 & (1-s)t_1 & -s \\
0 & 0 && \rotatebox[origin=c]{295}{\vdots} & (1-s)t_1 & -s & 0 \\
0 & \vdots & \rotatebox[origin=c]{295}{\vdots} & \rotatebox[origin=c]{295}{\vdots} & \rotatebox[origin=c]{295}{\vdots} & \rotatebox[origin=c]{295}{\vdots} & \vdots \\
\vdots & 0 & (1-s)t_1 & \rotatebox[origin=c]{295}{\vdots} & \rotatebox[origin=c]{295}{\vdots} && 0 \\
0 & (1-s)t_1 & -s & \rotatebox[origin=c]{295}{\vdots} &&& 0 \\
(n-s)t_1/n & -s & 0 & \cdots & 0 & 0 & (n-1)t_1^2/n
\end{array} 
\right]$}.
\end{align*}
Here, let us denote by
\begin{align*}
\varPhi_{\bm u}(x)=\sum_{k=0}^{s}h_{s-k}^{({\bm u})}x^{s-k}=\varPhi(u_0, \cdots, u_s; x) \ \left(=\sum_{k=0}^{s} h_{s-k}(u_0, \cdots, u_s)x^{s-k} \right)
\end{align*}
the characteristic polynomial of $\bar{B}_{\bm u}$. Note that since we have $h_{s-k}^{({\bm u})} \in \R[t_1]$ by the proof of Lemma \ref{lem5.3}, we have only to prove $h_{s-k}^{({\bm u})}$ is non-constant for any $k$ $(1 \leq  k \leq s)$.

By the above expression of $\bar{B}_{\bm u}$, we have \\[1mm]
$(a^{\prime \prime})$ if $s=2$,
\begin{align*}
\varPhi_{\bm u}(x)=x^2-\frac{(n-1)t_1^2-4}{n}x+\frac{(n-2)t_1^2-4n+8}{n},
\end{align*}
\normalsize
$(b^{\prime \prime})$ if $s \geq 3$,
\footnotesize
\begin{align*}
\varPhi_{\bm u}(x)
=
\begin{cases}
\scalebox{0.83}[1]
{$\left|
\begin{array}{ccccccccc}
x-s(n-s)/n &&&&&&&& -(n-s)t_1/n \\
& x &&&&&& (s-1)t_1 & s \\
&& \rotatebox[origin=c]{65}{\vdots} &&&& \rotatebox[origin=c]{295}{\vdots} & s & \\
&&& \rotatebox[origin=c]{65}{\vdots} && \rotatebox[origin=c]{295}{\vdots} & \rotatebox[origin=c]{295}{\vdots} & \\
&&&& x+(s-1)t_1 & s && \\
&&& \rotatebox[origin=c]{290}{\vdots} & s & x &&& \\
&& \rotatebox[origin=c]{290}{\vdots} & \rotatebox[origin=c]{290}{\vdots} &&& \rotatebox[origin=c]{65}{\vdots} && \\
& (s-1)t_1 & s &&&&& x & \\
-(n-s)t_1/n & s &&&&&&& x-(n-1)t_1^2/n
\end{array} 
\right|$} \vspace{2mm} \\ 
&\hspace{-16mm} \text{($s$ is odd)}, \\[3mm]
\scalebox{0.83}[1]
{$\left|
\begin{array}{cccccccc}
x-s(n-s)/n &  &  &  & && & -(n-s)t_1/n \\
& x &  &&&    & (s-1)t_1 & s \\
&  & \hspace{-7mm} \rotatebox[origin=c]{62}{\vdots} & &&  \rotatebox[origin=c]{300}{\vdots} & \hspace{-1mm} s & \\
&  &   & \hspace{-7mm} x  & \hspace{-5mm} (s-1)t_1 & \hspace{0mm} \rotatebox[origin=c]{300}{\vdots} & &  \\
&&& \hspace{-7mm} (s-1)t_1 & \hspace{-3mm} x+s &&& \\
&& \rotatebox[origin=c]{300}{\vdots} & \hspace{3mm} \rotatebox[origin=c]{300}{\vdots} && \rotatebox[origin=c]{65}{\vdots} && \\
 & (s-1)t_1 & s & && & x &  \\
-(n-s)t_1/n & s &   &  &&&  & x-(n-1)t_1^2/n
\end{array} 
\right|$} \vspace{2mm} \\
&\hspace{-17mm} \text{($s$ is even)}.
\end{cases}
\end{align*}
\normalsize
\begin{exa}
$(1)$ Put $s=7$ and $n=10$. Then, we have 
\begin{align*}
&g_{\bm u}(x)=x^7+t_1x+1, \ f_{\bm u}(t; x)=x^{10}+t(x^7+t_1x+1),
\end{align*}
\footnotesize
\begin{align*}
\scalebox{1}[1]
{$\varPhi_{\bm u}(x)$}&=
\scalebox{1}[1]
{$\left|
\begin{array}{ccccccc}
x-21/10 & 0 & 0 & 0 & 0 & 0 & -3t_1/10 \\
0 & x & 0 & 0 & 0 & 6t_1 & 7 \\
0 & 0 & x & 0 & 6t_1 & 7 & 0 \\
0 & 0 & 0 & x+6t_1 & 7 & 0 & 0 \\
0 & 0 & 6t_1 & 7 & x & 0 & 0 \\
0 & 6t_1 & 7 & 0 & 0 & x & 0 \\
-3t_1/10 & 7 & 0 & 0 & 0 & 0 & x-9t_1^2/10
\end{array} 
\right|$} \\[2mm]
=&
x^7+\left(-\frac{9}{10}t_1^2+6t_1-\frac{21}{10}\right)x^6+\left(-\frac{27}{5}t_1^3-\frac{351}{5}t_1^2-\frac{63}{5}t_1-147\right)x^5 \\[1mm]
&+\left(\frac{324}{5}t_1^4-\frac{2106}{5}t_1^3+\frac{1197}{5}t_1^2-588t_1+\frac{3087}{10}\right)x^4 \\[1mm]
&+\left(\frac{1944}{5}t_1^5+\frac{5832}{5}t_1^4+\frac{5859}{5}t_1^3+\frac{16758}{5}t_1^2+\frac{6174}{5}t_1+7203\right)x^3 \\[1mm]
&+\left(-\frac{5832}{5}t_1^6+\frac{34992}{5}t_1^5-\frac{21546}{5}t_1^4+\frac{50274}{5}t_1^3-\frac{95697}{10}t_1^2+14406t_1-\frac{151263}{10}\right)x^2 \\[1mm]
&+\biggl(-\frac{34992}{5}t_1^7+\frac{11664}{5}t_1^6-\frac{81648}{5}t_1^5+\frac{15876}{5}t_1^4-\frac{111132}{5}t_1^3+\frac{21609}{5}t_1^2-\frac{151263}{5}t_1 \\[1mm]
&-117649\biggr)x+\frac{69984}{5}t_1^7+\frac{2470629}{10}.
\end{align*}
\normalsize
$(2)$ Put $s=8$ and $n=12$. Then, we have 
\begin{align*}
&g_{\bm u}(x)=x^8+t_1x+1, \ f_{\bm u}(t ; x)=x^{12}+t(x^8+t_1x+1) 
\end{align*}
and
\footnotesize
\begin{align*}
\scalebox{1}[1]
{$\varPhi_{\bm u}(x)$}
&=
\scalebox{1}[1]
{$\left|
\begin{array}{cccccccc}
x-8/3 & 0 & 0 & 0 & 0 & 0 & 0 & -t_1/3 \\
0 & x & 0 & 0 & 0 & 0 & 7t_1 & 8 \\
0 & 0 & x & 0 & 0 & 7t_1 & 8 & 0 \\
0 & 0 & 0 & x & 7t_1 & 8 & 0 & 0 \\
0 & 0 & 0 & 7t_1 & x+8 & 0 & 0 & 0 \\
0 & 0 & 7t_1 & 8 & 0 & x & 0 & 0 \\
0 & 7t_1 & 8 & 0 & 0 & 0 & x & 0 \\
-t_1/3 & 8 & 0 & 0 & 0 & 0 & 0 & x-11t_1^2/12 
\end{array} 
\right|$} \\[2mm]
=&
x^8+\left(-\frac{11}{12}t_1^2+\frac{16}{3}\right)x^7+\left(-152t_1^2-\frac{640}{3}\right)x^6+\left(\frac{539}{4}t_1^4-256t_1^2-1024\right)x^5 \\[1mm]
&+\left(\frac{22736}{3}t_1^4+\frac{45824}{3}t_1^2+16384\right)x^4+\left(-\frac{26411}{4}t_1^6-\frac{22736}{3}t_1^4+\frac{31744}{3}t_1^2+65536\right)x^3 \\[1mm]
&+\left(-\frac{355348}{3}t_1^6-213248t_1^4-\frac{1064960}{3}t_1^2-524288\right)x^2+\biggl(\frac{1294139}{12}t_1^8+\frac{1075648}{3}t_1^6 \\[1mm]
&+\frac{1404928}{3}t_1^4+\frac{1835008}{3}t_1^2-\frac{4194304}{3}\biggr)x-\frac{823543}{3}t_1^8+\frac{16777216}{3}.
\end{align*}
\end{exa}
\normalsize
\begin{proof}[Proof of Lemma \ref{lem5.4}]
To prove Lemma \ref{lem5.4}, it is enough to prove $\deg h_{s-k}^{({\bm u})} \geq 1$ for any $k$ $(1 \leq k \leq s)$. This is clear for $s=2$ by $(a^{\prime \prime})$ and we suppose $s \geq 3$ hereafter. 
To prove $\deg h_{s-k}^{({\bm u})} \geq 1$ $(1 \leq k \leq s)$, let us compute the leading term of $h_{s-k}^{({\bm u})}$ $(\in \R[t_1])$. Then, since $h_{s-k}^{({\bm u})}$ is the coefficient of the term $h_{s-k}^{({\bm u})}x^{s-k}$ of the characteristic polynomial $\varPhi_{\bm u}(x)$, we need to maximize the degree in $t_1$ when we take `$s-k$' $x$ and the remaining $k$ elements from $\R[t_1]$. \\[1.5mm]
$(a)$ Suppose $s$ is odd. Let us divide the case into three other subcases. \\[1.5mm]
$(a1)$ Suppose $k$ is odd and $1 \leq k \leq s-2$. \\
In this case, the degree of the leading term of $h_{s-k}^{({\bm u})}$ is $k+1$. In fact, it is obtained by taking 
\begin{enumerate}
\item[$(a11)$]
$-(n-1)t_1^2/n$ from the $(s, s)$ entry $x-(n-1)t_1^2/n$, 
\item[$(a12)$]
`$k-1$' $(s-1)t_1$ from entries of the form $(i, s+1-i)$ $(2 \leq i \leq s-1)$. 
\end{enumerate}

First, suppose we take the $(s, s)$ entry $x-(n-1)t_1^2/n$ from the $s$-th row. Then we must take the $(1, 1)$ entry from the first row. Next, let us proceed to the $(s-1)$-th row. If we take the $(s-1, s-1)$ entry $x$ from the $(s-1)$-th row, then we must also take $x$ from the second row, while if we take $(s-1)t_1$ from the $(s-1)$-th row, then we must also take $(s-1)t_1$ from the second row. The situation is the same for the $(s-2)$-th row, the $(s-3)$-th row ... and so on, which implies that $(s-1)t_1$ must occur in pair. 

Hence, the leading term of $h_{s-k}^{({\bm u})}$ is
\begin{align*}
-\frac{n-1}{n}t_1^2 \cdot \binom{(s-3)/2}{(k-1)/2}\{(-1)\cdot(s-1)^2t_1^2\}^{(k-1)/2} \hspace{3mm} \left(\binom{n}{0}=1  \ (n \geq 0)\right)
\end{align*}
and the degree of this term is $k+1$ $(\geq 2)$. \\[1mm]
$(a2)$ Suppose $k$ is odd and $k=s$. \\
If $k=s$, $h_{s-k}^{({\bm u})}=h_0^{({\bm u})}$ is the constant term of $\varPhi_{\bm u}(x)$. In this case, the degree of the leading term of $h_0^{({\bm u})}$ is $s$. In fact, it is obtained by taking 
\begin{enumerate}
\item[$(a21)$]
$-(n-1)t_1^2/n$ from the $(s, s)$ entry $x-(n-1)t_1^2/n$, \vspace{1mm}
\item[$(a22)$]
If $s \geq 5$ $(\Leftrightarrow (s,k)\neq (3,3))$, `$(s-3)/2$' pairs of $(s-1)t_1$ from entries of the form $(i, s+1-i)$ $(2 \leq i \leq (s-1)/2, \ (s+3)/2 \leq i \leq s-1)$, \vspace{1mm}
\item[$(a23)$]
$(s-1)t_1$ from the $((s+1)/2, (s+1)/2)$ entry $x+(s-1)t_1$, \vspace{1mm}
\item[$(a24)$]
$-s(n-s)/n$ from the $(1, 1)$ entry $x-s(n-s)/n$
\end{enumerate}
or by taking 
\begin{enumerate}
\item[$(a25)$]
all anti-diagonal entries. 
\end{enumerate}
Therefore, the leading term of $h_0^{({\bm u})}$ is
\begin{align*}
-&\dfrac{n-1}{n}t_1^2 \cdot \{(-1)\cdot(s-1)^2t_1^2\}^{(s-3)/2}\cdot(s-1)t_1\cdot\left(-\frac{s(n-s)}{n}\right) \\
&\hspace{20mm}+(-1)\cdot\left(-\dfrac{n-s}{n}t_1\right)^2\cdot \{(-1)\cdot(s-1)^2t_1^2\}^{(s-3)/2}\cdot(s-1)t_1 \\
&=\frac{(n-s)(s-1)}{n} \cdot (-1)^{(s-3)/2}(s-1)^{s-2}t_1^s \\
&=(-1)^{(s-3)/2}\dfrac{(n-s)(s-1)^{s-1}}{n}t_1^{s} 
\end{align*}  
for any $s$ $(s \geq 3)$ and the degree of this term is $s$. \\[1.5mm]
$(a3)$ Suppose $k$ is even. \\
In this case, we have $2 \leq k \leq s-1$ and the degree of the leading term of $h_{s-k}^{({\bm u})}$ is $k+1$. In fact, it is obtained by taking 
\begin{enumerate}
\item[$(a31)$]
$-(n-1)t_1^2/n$ from the $(s, s)$ entry $x-(n-1)t_1^2/n$, \vspace{1.5mm}
\item[$(a32)$]
If $s \geq 5$ $(\Leftrightarrow (s,k)\neq (3,2))$, `$(k-2)/2$' pairs of $(s-1)t_1$ from entries of the form $(i, s+1-i)$ $(2 \leq i \leq (s-1)/2, \ (s+3)/2 \leq i \leq s-1)$, \vspace{1.5mm}
\item[$(a33)$]
$(s-1)t_1$ from the $((s+1)/2, (s+1)/2)$ entry $x+(s-1)t_1$.
\end{enumerate}
Therefore, the leading term of $h_{s-k}^{({\bm u})}$ is 
\begin{align*}
-\frac{n-1}{n}t_1^2\cdot \binom{(s-3)/2}{(k-2)/2} \{(-1) \cdot (s-1)^2t_1^2\}^{(k-2)/2}\cdot(s-1)t_1 
\end{align*}
for any $s$ $(s \geq 3)$ and the degree of this term is $k+1$ $(\geq 3)$. \\[1.5mm]
$(b)$ Suppose $s$ is even ($s \geq 4$). We also divide this case into three other subcases. \\[1.5mm]
$(b1)$ Suppose $k$ is odd. \\
In this case, we have $1 \leq k \leq s-1$ and the degree of the leading term of $h_{s-k}^{({\bm u})}$ is $k+1$. In fact, it is obtained by taking 
\begin{enumerate}
\item[$(b11)$]
$-(n-1)t_1^2/n$ from the $(s, s)$ entry $x-(n-1)t_1^2/n$,
\item[$(b12)$]
`$(k-1)/2$' pairs of $(s-1)t_1$ from entries of the form $(i, s+1-i)$ $(2 \leq i \leq s-1)$.
\end{enumerate}
Therefore, the leading term of $h_{s-k}^{({\bm u})}$ is
\begin{align*}
-\frac{n-1}{n}t_1^2 \cdot \binom{(s-2)/2}{(k-1)/2} \{ (-1) \cdot (s-1)^2t_1^2 \}^{(k-1)/2}
\end{align*}
and the degree of this term is $k+1$ $(\geq 2)$. \\[1.5mm]
$(b2)$ Suppose $k$ is even and $2 \leq k \leq s-2$. \\
In this case, the degree of the leading term of $h_{s-k}^{({\bm u})}$ is $k$. In fact, it is obtained by taking
\begin{enumerate}
\item[$(b21)$]
$-(n-1)t_1^2/n$ from the $(s, s)$ entry $x-(n-1)t_1^2/n$,
\item[$(b22)$]
`$(k-2)/2$' pairs of $(s-1)t_1$ from entries of the form $(i, s+1-i)$ $(2 \leq i \leq s-1)$,
\item[$(b23)$]
$-s(n-s)/n$ from the $(1, 1)$ entry $x-s(n-s)/n$
\end{enumerate}
or by taking
\begin{enumerate}
\item[$(b24)$]
$-(n-1)t_1^2/n$ from the $(s, s)$ entry $x-(n-1)t_1^2/n$, \vspace{1.5mm}
\item[$(b25)$]
I f $s \geq 6$ $(\Leftrightarrow (s, k) \neq (4. 2))$, `$(k-2)/2$' pairs of $(s-1)t_1$ from entries of the form $(i, s+1-i)$ $(2 \leq i \leq (s-2)/2, \ (s+4)/2 \leq i \leq s-1)$, \vspace{1.5mm}
\item[$(b26)$]
$s$ from the $((s+2)/2, (s+2)/2)$ entry $x+s$
\end{enumerate}
or by taking
\begin{enumerate}
\item[$(b27)$]
`$k/2$' pairs of $(s-1)t_1$ from entries of the form $(i, s+1-i)$ $(2 \leq i \leq s-1)$ 
\end{enumerate}
or by taking
\begin{enumerate}
\item[$(b28)$]
One pair of $-(n-s)t_1/n$ from the $(1,s)$ and the $(s,1)$ entry,
\item[$(b29)$]
`$(k-2)/2$' pairs of $(s-1)t_1$ from entries of the form $(i, s+1-i)$ $(2 \leq i \leq s-1)$. 
\end{enumerate}
Here, note that if we take the $(s,1)$ entry $-(n-s)t_1/n$ from the $s$-th row, we must also take the $(1,s)$ entry $-(n-s)t_1/n$ from the first row.

Therefore, the leading term of $h_{s-k}^{(\bm u)}$ is \\[0mm]
\footnotesize
\begin{align*}
\displaystyle -&\dfrac{n-1}{n}t_1^2\cdot\binom{(s-2)/2}{(k-2)/2} \{(-1)\cdot (s-1)^2t_1^2 \}^{(k-2)/2}\cdot\left( -\frac{s(n-s)}{n} \right) \\[1.5mm]
&-\dfrac{n-1}{n}t_1^2\cdot\binom{(s-4)/2}{(k-2)/2} \{(-1)\cdot (s-1)^2t_1^2 \}^{(k-2)/2}\cdot s+\binom{(s-2)/2}{k/2}\{ (-1)\cdot(s-1)^2t_1^2 \}^{k/2} \\[1.5mm]
&+\left( (-1)\cdot\frac{\{-(n-s)\}^2}{n^2}t_1^2\right) \cdot \binom{(s-2)/2}{(k-2)/2} \{(-1)\cdot (s-1)^2t_1^2\}^{(k-2)/2}  \\[2mm]
&\hspace{-3mm}=\biggl( \frac{s(n-s)(n-1)}{n^2}\binom{(s-2)/2}{(k-2)/2}-\frac{s(n-1)}{n}\binom{(s-4)/2}{(k-2)/2} \\[1.5mm]
&\hspace{24mm}-(s-1)^2\binom{(s-2)/2}{k/2}-\frac{(n-s)^2}{n^2}\binom{(s-2)/2}{(k-2)/2} \biggr)\{(-1)\cdot (s-1)^2t_1^2\}^{(k-2)/2}t_1^2.
\end{align*}
\normalsize
for any $s$ $(s \geq 4)$. Then, since
\footnotesize
\begin{align*}
&\binom{(s-4)/2}{(k-2)/2}=\frac{s-k}{s-2}\binom{(s-2)/2}{(k-2)/2}, \ \binom{(s-2)/2}{k/2}=\frac{s-k}{k}\binom{(s-2)/2}{(k-2)/2}, 
\end{align*}
\normalsize
we have \\[0mm]
\footnotesize
\begin{align} \label{eq5.61}
&\frac{s(n-s)(n-1)}{n^2}\binom{(s-2)/2}{(k-2)/2}-\frac{s(n-1)}{n}\binom{(s-4)/2}{(k-2)/2} \\[1.5mm] \notag
&\hspace{36.3mm}-(s-1)^2\binom{(s-2)/2}{k/2}-\frac{(n-s)^2}{n^2}\binom{(s-2)/2}{(k-2)/2} \\[2mm] \notag
&=\left(\frac{s(n-s)(n-1)}{n^2}-\frac{s(s-k)(n-1)}{n(s-2)}-\frac{(s-1)^2(s-k)}{k}-\frac{(n-s)^2}{n^2} \right)\binom{(s-2)/2}{(k-2)/2} \\[2mm] \notag 
\vspace{2mm}
&=\frac{s\{\left(k(k+s^2-4s+2)-s^3+4s^2-5s+2\right)n-k(k+s^2-4s+2)\}}{nk(s-2)}\binom{(s-2)/2}{(k-2)/2}. \notag
\end{align} 
\normalsize
Hence, if the above value becomes zero, we have
\begin{align*}
\left(k(k+s^2-4s+2)-s^3+4s^2-5s+2\right)n-k(k+s^2-4s+2)=0,
\end{align*}
which implies
\begin{align} \label{eq5.7}
k(k+s^2-4s+2)=0, \ -s^3+4s^2-5s+2=0
\end{align}
or
\begin{align} \label{eq5.8}
n=\frac{k(k+s^2-4s+2)}{k(k+s^2-4s+2)-s^3+4s^2-5s+2}.
\end{align}
Here, \eqref{eq5.7} is impossible since $-s^3+4s^2-5s+2=-(s-1)^2(s-2)$ and $s \geq 4$. Also, \eqref{eq5.8} is impossible since, for any $s \geq 4$ and $2 \leq k \leq s-2$, we have
\begin{align*}
k(k+s^2-4s+2) &\geq 2(2+s^2-4s+2) \geq 2(s-2)^2> 0
\end{align*}
and
\begin{align*}
&k(k+s^2-4s+2)-s^3+4s^2-5s+2 \\
&\leq (s-2)\{ (s-2)+s^2-4s+2 \}-s^3+4s^2-5s+2 \\
&=-s^2+s+2 \\
&=-(s+1)(s-2)<0,
\end{align*}
which implies $n<0$, a contradiction. Thus, the above value \eqref{eq5.61} is non-zero and the degree of the leading term of $h_{s-k}^{({\bm u})}$ is $k$. \\[1.5mm]
$(b3)$ Suppose $k$ is even and $k=s$. \\
If $k=s$, $h_{s-k}^{({\bm u})}=h_0^{({\bm u})}$ is the constant term of $\varPhi_{\bm u}(x)$. In this case, the degree of the leading term of $h_0^{({\bm u})}$ is $s$. In fact, it is obtained by taking 
\begin{enumerate}
\item[$(b31)$]
$-(n-1)t_1^2/n$ from the $(s, s)$ entry $x-(n-1)t_1^2/n$, 
\item[$(b32)$]
`$(s-2)/2$' pairs of $(s-1)t_1$ from entries of the form $(i, s+1-i)$ $(2 \leq i \leq s-1)$,
\item[$(b33)$]
$-s(n-s)/n$ from the $(1, 1)$ entry $x-s(n-s)/n$
\end{enumerate}
or by taking 
\begin{enumerate}
\item[$(b34)$]
all anti-diagonal entries. 
\end{enumerate}
Therefore, the leading term of $h_0^{({\bm u})}$ is 
\begin{align*}
-&\dfrac{n-1}{n}t_1^2 \cdot \{(-1)\cdot(s-1)^2t_1^2\}^{(s-2)/2}\cdot \left( -\frac{s(n-s)}{n} \right) \\
&\hspace{21.5mm}+(-1)\cdot\left(-\dfrac{n-s}{n}t_1\right)^2\cdot \{(-1)\cdot(s-1)^2t_1^2\}^{(s-2)/2} \\
&=(-1)^{(s-2)/2}\dfrac{(n-s)(s-1)^{s-1}}{n}t_1^{s}
\end{align*}  
and the degree of this term is $s$ $(s \geq 4)$.
\end{proof}
\begin{lem} \label{lem5.5}
Let ${\bm v}=(v_0, \cdots, v_s) \in \R^{s+1}$ be a real vector and $n$ $(>s)$ be an integer. Put
\begin{align*}
P_{\bm v}(t)=\det M_n(f_{\bm v}(t ; x))=\det M_n(f^{(n)}(v_0, \cdots, v_s, t; x))
\end{align*}
and $\alpha_{\bm v}=\max \{ \alpha \in \R \mid P_{\bm v}(\alpha)=0 \}$. If there exists a real number $\rho_0$ $(>\alpha_{\bm v})$ such that $N_{f_{\bm v}(\xi ; x)}=\gamma_{0}$ for any $\xi>\rho_0$, we have $N_{f_{\bm v}(\xi ; x)}=\gamma_{0}$ for any $\xi>\alpha_{\bm v}$.
\end{lem}
\begin{proof}
Put $A_{\bm v}(t)=M_n(f_{\bm v}(t ; x))$. Then, by Proposition \ref{prop5.2}, we have $\gamma_0=\sigma(A_{\bm v}(\xi))$ for any $\xi>\rho_0$. Let us also put
\[
R=\{ \rho \in \R \mid \rho>\alpha_{\bm v}, \ \sigma(A_{\bm v}(\xi))=\gamma_0 \ \text{for any $\xi>\rho$} \}.
\] 
Since $R$ is a nonempty set ($\rho_0 \in R$) having a lower bound $\alpha_{\bm v}$, $R$ has the infimum $\rho_{\bm v}$; $\rho_{\bm v}=\inf R$.
Then, it is enough to prove $\rho_{\bm v}=\alpha_{\bm v}$. Here, suppose to the contrary that $\rho_{\bm v}>\alpha_{\bm v}$ and we denote by 
\begin{align*}
\Omega_{\bm v}(t; x)=\sum_{k=0}^n \omega_k(t) x^k \in \R(t)[x] 
\end{align*}
the characteristic polynomial of $A_{\bm v}(t)$. Note that $\omega_k(t) \in \R[t]$ $(0 \leq k \leq n)$ and for any $\xi>\alpha_{\bm v}$, $\Omega_{\bm v}(\xi; x)$ has $n$ non-zero real roots (counted with multiplicity) since $A_{\bm v}(\xi)$ is symmetric and $\det A_{\bm v}(\xi) \neq 0$. Then, by Proposition \ref{prop5.3}, there exists a positive real number $\delta$ such that $\rho_{\bm v}-\delta>\alpha_{\bm v}$ and for any $\xi \in [\rho_{\bm v}-\delta, \rho_{\bm v}+\delta]$, $\Omega_{\bm v}(\xi; x)$ has the same number of positive and hence negative real roots with $\Omega_{\bm v}(\rho_{\bm v}; x)$. On the other hand, since $\rho_{\bm v}=\inf R$, there exist real numbers $\xi_{+}$ $(\rho_{\bm v}<\xi_+<\rho_{\bm v}+\delta)$ and $\xi_-$ $(\rho_{\bm v}-\delta<\xi_-<\rho_{\bm v})$ such that $\sigma(A_{\bm v}(\xi_+))\neq \sigma(A_{\bm v}(\xi_-))$, which implies $\Omega_{\bm v}(\xi_+; x)$ and $\Omega_{\bm v}(\xi_-; x)$ have different number of positive and hence negative real roots. This is a contradiction and we have $\rho_{\bm v}=\alpha_{\bm v}$. 
\end{proof}

\subsection{Proof of Theorem \ref{thm5.2}}
Let ${\bm r}=(r_0, \cdots, r_s) \in \R^{s+1}$ be the vector as in Theorem \ref{thm5.2} and put
\begin{align*}
n_0
=
\begin{cases}
(n-s+1)/2, & \text{$n-s-1$ $:$ even} \\
(n-s+2)/2, & \text{$n-s-1$ $:$ odd}.
\end{cases}
\end{align*}
When $n-s \geq 2$, we inductively define the matrix $A_{\bm r}(t)_{k}=(a_{ij}^{({\bm r})}(t)_k)_{1\leq i,j \leq n}$ $(2 \leq k \leq n-s)$ as the matrix obtained from $A_{\bm r}(t)_{k-1}$ by sweeping out the entries of the $k$-th row ($k$-th column) by the $(k,l_0-k)$ entry $-(n-s)r_st$ ($(l_0-k,k)$ entry $-(n-s)r_st$).
That is, 
we define  $A_{\bm r}(t)_{k}={}^tS_{\bm r}(t)_kA_{\bm r}(t)_{k-1}S_{\bm r}(t)_k$,
where
\begin{align*}
S_{\bm r}(t)_k=
\begin{cases}
\displaystyle \prod_{m=l_0-k+1}^{n}R_n\left(l_0-k,m;-\dfrac{a_{km}^{({\bm r})}(t)_{k-1}}{-(n-s)r_st}\right) &\hspace{-24mm}\text{($2 \leq k \leq n_0$)} \\[3mm]
\displaystyle R_n\left(l_0-k,k;-\dfrac{a_{kk}^{({\bm r})}(t)_{k-1}}{-2(n-s)r_st}\right)\prod_{m=k+1}^{n}R_n\left(l_0-k,m;-\dfrac{a_{km}^{({\bm r})}(t)_{k-1}}{-(n-s)r_st}\right) \\[3mm]
\hfill \text{($n_0 < k \leq n-s$)}.
\end{cases}
\end{align*} 
Then, if $n-s \geq 1$, we can express the matrix $A_{\bm r}(t)_{n-s}$ as follows;
\begin{align*}
A_{\bm r}(t)_{n-s}=
\scalebox{0.99}[0.99]
{$\left[
\begin{array}{cccc|cccc}
1 & 0 & \dots & 0 &&&& \\
0 & 0 & \dots & -(n-s)r_st &&&& \\
\vdots & \vdots & \rotatebox[origin=c]{295}{\vdots} & 0 &\multicolumn{4}{c}{\raisebox{3pt}[0pt][0pt]{\huge{$O$}}} \\
0 & -(n-s)r_st & 0 & 0 &&&& \\ \cline{1-8}
&&&&&&& \\
\multicolumn{3}{c}{\raisebox{-10pt}[0pt][0pt]{\hspace{24mm} \huge{$O$}}} && \multicolumn{4}{c}{\raisebox{-10pt}[0pt][0pt]{$C_{\bm r}(t)_{n-s}$}} \\
&&&&&&& \\  
&&&&&&& \\
\end{array}
\right] 
$}.
\end{align*} 
Note that $a_{km}^{({\bm r})}(t)_{k-1}$ and $a_{kk}^{(\bm r)}(t)_{k-1}$ appearing in $S_{\bm r}(t)_k$ are degree $1$ monomials in $t$ and hence the numbers
$-a_{km}^{({\bm r})}(t)_{k-1}/(-(n-s)r_st)$, $-a_{kk}^{({\bm r})}(t)_{k-1}/(-2(n-s)r_st)$
appearing in $S_{\bm r}(t)_k$ are just real numbers. Therefore, the entries of the $s\times s$ symmetric matrix $C_{\bm r}(t)_{n-s}=(c_{ij}^{({\bm r})}(t)_{n-s})_{1 \leq i,j \leq s}$ $(n-s \geq 1)$ are of the form 
\begin{align} \label{eq5.9}
c_{ij}^{({\bm r})}(t)_{n-s}=\bar{b}_{ij}^{({\bm r})}t^2+\bar{\lambda}_{ij}^{({\bm r})}t \hspace{3mm} (\bar{\lambda}_{ij}^{(\bm r)} \in \R).
\end{align}
Moreover, since the matrix
\begin{align*}
D_{\bm r}(t)_{n-s}
=
\scalebox{1}[0.99]
{$\left[
\begin{array}{cccc}
1 & 0 & \dots & 0 \\
0 & 0 & \dots & -(n-s)r_st \\
\vdots & \vdots & \rotatebox[origin=c]{295}{\vdots} & 0 \\
0 & -(n-s)r_st & 0 & 0 
\end{array}
\right]$}
\end{align*}
is equivalent to the matrix
{\footnotesize
\begin{align*}
\bar{D}_{\bm r}(t)_{n-s}
=
\begin{cases}
&\scalebox{1}[1]
{$\left[
\begin{array}{cccccc}
\multicolumn{1}{c|}{1} &&&&& \\ \cline{1-3}
& \multicolumn{1}{|c}{0} & \multicolumn{1}{c|}{-(n-s)r_st} &&& \\
& \multicolumn{1}{|c}{-(n-s)r_st} & \multicolumn{1}{c|}{0} &&& \\ \cline{2-3}
&&& \rotatebox[origin=c]{240}{\vdots} && \\ \cline{5-6}
&&&& \multicolumn{1}{|c}{0} & -(n-s)r_st \\
&&&& \multicolumn{1}{|c}{-(n-s)r_st} & 0 \\ 
\end{array}
\right]$} \\[0mm]
&\hfill \text{\normalsize ($n-s$ $:$ odd)} \\[2mm]
&\scalebox{0.88}[1]
{$\left[
\begin{array}{ccccccc}
\multicolumn{1}{c|}{1} &&&&&& \\ \cline{1-2}
& \multicolumn{1}{|c|}{-(n-s)r_st} &&&&& \\ \cline{2-4}
&& \multicolumn{1}{|c}{0} & \multicolumn{1}{c|}{-(n-s)r_st} &&& \\
&& \multicolumn{1}{|c}{-(n-s)r_st} & \multicolumn{1}{c|}{0} &&& \\ \cline{3-4}
&&&& \rotatebox[origin=c]{240}{\vdots} && \\ \cline{6-7}
&&&&& \multicolumn{1}{|c}{0} & -(n-s)r_st \\ 
&&&&& \multicolumn{1}{|c}{-(n-s)r_st} & 0 \\ 
\end{array}
\right]$} \\[2mm]
&\hspace{82mm} \text{\normalsize ($n-s$ $:$ even)}
\end{cases}
\end{align*}
}
\normalsize
over $\R$, we have
\begin{align} \label{eq5.10}
\sigma(D_{\bm r}(\xi)_{n-s})=\sigma(\bar{D}_{\bm r}(\xi)_{n-s})=
\begin{cases}
1 & \text{$n-s$ $:$ odd} \\
0 & \text{$n-s$ $:$ even, $r_s>0$} \\
2 & \text{$n-s$ $:$ even, $r_s<0$}
\end{cases}
\end{align}
for any real number $\xi>\alpha_{\bm r}$ $(\geq 0)$. Here, note that since $P_{\bm r}(0)=0$, we have $\alpha_{\bm r}\geq 0$.

Next, let $\varPhi_{\bm r}(t;x)$, $\varPsi_{\bm r}(t;x)$ be characteristic polynomials of $\bar{B}_{\bm r}(t)$, $C_{\bm r}(t)_{n-s}$, respectively. Then, by equations \eqref{eq5.54} and \eqref{eq5.9}, we have
\begin{align*}
\varPhi_{\bm r}(t ; x)&=x^s+h_{s-1}^{({\bm r})}t^2x^{s-1}+\cdots+h_{1}^{({\bm r})}t^{2s-2}x +h_{0}^{({\bm r})}t^{2s} \\ 
&\hspace{45mm}\left(h_{s-k}^{({\bm r})}=h_{s-k}(r_0, \cdots, r_s)\in \R \ (1 \leq k \leq s)\right), \\
\varPsi_{\bm r}(t ; x)&=x^s+\left(h_{s-1}^{({\bm r})}t^2+\psi_{s-1}(t)\right)x^{s-1}+\cdots \\
&\hspace{45mm}+\left(h_{1}^{({\bm r})}t^{2s-2}+\psi_{1}(t)\right)x+\left(h_{0}^{({\bm r})}t^{2s}+\psi_0(t)\right) \\
&\hspace{25mm} \left(\psi_0(t), \cdots, \psi_{s-1}(t) \in \R[t], \deg \psi_{s-k}(t) < 2k \ (1\leq k \leq s)\right).
\end{align*}
Here, let us divide the proof into next two cases. \\[2mm]
(i) The case $h_0^{({\bm r})}h_1^{({\bm r})} \cdots h_{s-1}^{({\bm r})} \neq 0$. \\[2mm]
In this case, we have
\begin{align*}
\varPsi_{\bm r}(t ; x)&=x^s+h_{s-1}^{({\bm r})}t^2\left( 1+\dfrac{\psi_{s-1}(t)}{h_{s-1}^{(\bm r)}t^2} \right)x^{s-1}+\cdots \\
&\hspace{30mm}+h_{1}^{({\bm r})}t^{2s-2}\left(1+\dfrac{\psi_{1}(t)}{h_{1}^{(\bm r)}t^{2s-2}}\right)x+h_{0}^{({\bm r})}t^{2s}\left(1+\dfrac{\psi_0(t)}{h_0^{(\bm r)}t^{2s}}\right)
\end{align*}
and $1+\psi_{s-k}(t)\bigl/h_{s-k}^{(\bm r)}t^{2k} \to 1$ $(t \to \infty)$ for any $k$ $(1 \leq k \leq s)$. Moreover, since $h_0^{({\bm r})}h_1^{({\bm r})} \cdots h_{s-1}^{({\bm r})} \neq 0$, we have $h_0^{({\bm r})}\neq 0$, which implies that for any non-zero real number $\xi$, $\varPhi_{\bm r}(\xi ; x)$ have $s$ non-zero real roots (counted with multiplicity).
Thus, there exists a real number $\rho_0$ $(>\alpha_{\bm r})$ such that for any real number $\xi>\rho_0$, $\varPsi_{\bm r}(\xi ; x)$ have the same number of positive (hence also negative) real roots with $\varPhi_{\bm r}(\xi ; x)$ by Proposition \ref{prop5.3}, which implies $\sigma(C_{\bm r}(\xi)_{n-s})=\sigma(\bar{B}_{\bm r}(\xi))$ and hence $\sigma(C_{\bm r}(\xi)_{n-s})=N_{g_{\bm r}}=\gamma$ $(\xi>\rho_0)$ by Lemma \ref{lem5.2}. 
Then, by the equation \eqref{eq5.10}, 
we have
\begin{align*}
\sigma(A_{\bm r}(\xi)_{n-s})=
\begin{cases}
\gamma+1 & \text{$n-s$ $:$ odd} \\
\gamma & \text{$n-s$ $:$ even, $r_s>0$} \\
\gamma+2 & \text{$n-s$ $:$ even, $r_s<0$}
\end{cases} 
\end{align*}
for any $\xi>\rho_0$, which implies
\begin{align*}
N_{f_{\bm r}(\xi ; x)}=\sigma(A_{\bm r}(\xi))=
\begin{cases}
\gamma+1 & \text{$n-s$ $:$ odd} \\
\gamma & \text{$n-s$ $:$ even, $r_s>0$} \\
\gamma+2 & \text{$n-s$ $:$ even, $r_s<0$}
\end{cases} 
\end{align*}
for any $\xi>\rho_0$ since $A_{\bm r}(\xi)$ and $A_{\bm r}(\xi)_{n-s}$ are equivalent over $\R$. Hence, by Lemma \ref{lem5.5}, we have 
\begin{align*}
N_{f_{\bm r}(\xi ; x)}=
\begin{cases}
\gamma+1 & \text{$n-s$ $:$ odd} \\
\gamma & \text{$n-s$ $:$ even, $r_s>0$} \\
\gamma+2 & \text{$n-s$ $:$ even, $r_s<0$}
\end{cases} 
\end{align*}
for any $\xi>\alpha_{\bm r}$. \\[2mm]
(ii) General case. \\[-3mm]

Let $\varepsilon_0$ be a positive real number and for any vector ${\bm v} \in \R^{s+1}$, set 
\begin{align*}
\alpha_{\bm v}^{\prime}=\max\{ |\alpha| \mid \alpha \in \C, P_{\bm v}(\alpha)=0 \}. 
\end{align*}
Clearly, we have $\alpha_{\bm v}^{\prime} \geq \alpha_{\bm v}$ for any ${\bm v} \in \R^{s+1}$. 
Here, let us put $\rho_0^{\prime}=\alpha_{\bm r}^{\prime}+\varepsilon_0$. Then, by Lemma \ref{lem5.5}, it is enough to prove the next claim. 

\begin{clm}
For any real number $\xi>\rho_0^{\prime}$, we have 
\begin{align*}
N_{f_{\bm r}(\xi ; x)}
=
\begin{cases}
\gamma+1 & \text{$n-s$ $:$ odd} \\
\gamma & \text{$n-s$ $:$ even, $r_s>0$} \\
\gamma+2 & \text{$n-s$ $:$ even, $r_s<0$}.
\end{cases}
\end{align*}
\end{clm}

\begin{proof}
By the assumption that $g_{\bm r}(x)$ is a separable polynomial of degree $s$ and the fact that the non-real roots must occur in pair with its complex conjugate, there exists a real number $\delta_0$ such that for any vector ${\bm v}=(v_0, \cdots, v_s) \in \R^{s+1}$ satisfying 
$|{\bm r}-{\bm v}|_0=\max_{0 \leq k \leq s}\{ |r_k-v_k| \}<\delta_0$, 
$g_{\bm v}(x)$ is also a degree $s$ separable polynomial satisfying  $N_{g_{\bm v}}=N_{g_{\bm r}}=\gamma$ by Proposition \ref{prop5.3}. 
\begin{align*}
\text{(S1)} \hspace{1.5mm} &\text{If a vector ${\bm v} \in \R^{s+1}$ satisfies $|{\bm r}-{\bm v}|_0<\delta_0$, then $g_{\bm v}(x)$ is also a degree $s$} \\ 
&\text{separable polynomial satisfying $N_{g_{\bm v}}=N_{g_{\bm r}}=\gamma$. }
\end{align*}
Next, we put
\begin{align*}
P(t)=\sum_{k \geq 0} x_k(t_0, \cdots, t_s)t^k=\det A(t) \ (A(t)=A(t_0, \cdots, t_s, t)) 
\end{align*}
and let us consider $P(t)$ as a polynomial over $E_1=\R(t_0, \cdots, t_s)$ in $t$. Then, since $x_k(t_0, \cdots, t_s) \in \R[t_0, \cdots, t_s]$ for any $k \geq 0$, there exists a real number $\delta_1>0$ such that for any vector ${\bm v} \in \R^{s+1}$ satisfying $|{\bm r}-{\bm v}|_0<\delta_1$, we have $|\alpha_{\bm r}^{\prime}-\alpha_{\bm v}^{\prime}|<\varepsilon_0$ by Proposition \ref{prop5.3};
\begin{align*}
\text{(S2) If a vector ${\bm v} \in \R^{s+1}$ satisfies $|{\bm r}-{\bm v}|_0<\delta_1$, we have $|\alpha_{\bm r}^{\prime}-\alpha_{\bm v}^{\prime}|<\varepsilon_0$.}
\end{align*}
Here, let $\xi$ be any real number such that $\xi>\rho_0^{\prime}=\alpha_{\bm r}^{\prime}+\varepsilon_0$ and let 
\begin{align*}
\Omega(t_0, \cdots, t_s, \xi ; x)=\sum_{k=0}^n y_k(t_0, \cdots, t_s)x^k \in E_1[x]
\end{align*}
be the characteristic polynomial of the Bezoutian 
\begin{align*}
A(t_0, \cdots, t_s, \xi ; x)=M_n(f^{(n)}(t_0, \cdots, t_s, \xi ; x), f^{(n)}(t_0, \cdots, t_s, \xi ; x)^{\prime}).
\end{align*}
Here, $f^{(n)}(t_0, \cdots, t_s, \xi ; x)^{\prime}$ is the derivative of
\begin{align*}
f^{(n)}(t_0, \cdots, t_s, \xi ; x)=\sum_{k=0}^n z_k(t_0, \cdots, t_s)x^k \in E_1[x]
\end{align*}
with respect to $x$. Then, since $z_k(t_0, \cdots, t_s) \in \R[t_0, \cdots, t_s]$ $(0 \leq k \leq n)$, we also have $y_k(t_0, \cdots, t_s) \in \R[t_0, \cdots, t_s]$ $(0 \leq k \leq n)$. Moreover, since $\xi>\rho_0^{\prime}>\alpha_{\bm r}$, we have $\det A_{\bm r}(\xi)=\det A(r_0, \cdots, r_s, \xi)\neq 0$. 

By these arguments, we can also deduce that there exists a positive real number $\delta_2$ such that for any vector ${\bm v} \in \R^{s+1}$ satisfying $|{\bm r}-{\bm v}|_0<\delta_2$, the characteristic polynomial $\Omega_{\bm v}(\xi ; x)$ have the same number of positive and hence negative real roots with $\Omega_{\bm r}(\xi ; x)$ (counted with multiplicity), which implies
$N_{f_{\bm r}(\xi ; x)}=\sigma(A_{\bm r}(\xi))=\sigma(A_{\bm v}(\xi))=N_{f_{\bm v}(\xi ; x)}$.
\begin{align*}
\text{(S3) If a vector ${\bm v} \in \R^{s+1}$ satisfies $|{\bm r}-{\bm v}|_0<\delta_2$, we have $N_{f_{\bm r}(\xi ; x)}=N_{f_{\bm v}(\xi ; x)}$.}
\end{align*}
Put $\delta=\min\{ \delta_0, \delta_1, \delta_2 \}>0$. Then, there exists a vector ${\bm w}=(w_0, \cdots, w_s) \in \R^{s+1}$ such that 
\begin{center}
$(a)$ $|{\bm r}-{\bm w}|_0<\delta$, $(b)$ $h_0^{({\bm w})}h_1^{({\bm w})}\cdots h_{s-1}^{({\bm w})}\neq 0$. 
\end{center}
Here, we put $h_{s-k}^{({\bm w})}=h_{s-k}(w_0, \cdots, w_s)$ for any $k$ $(1 \leq k \leq s)$. 
In fact, since $h_{s-k}(t_0, \cdots, t_s)$ is a non-zero polynomial for any $k$ $(1 \leq k \leq s)$ by Lemma \ref{lem5.3}, the product $\prod_{k=1}^s h_{s-k}(t_0, \cdots, t_s)$ is also non-zero, which implies that there exists a vector ${\bm w} \in \R^{s+1}$ satisfying $(a)$ and $(b)$.

Let ${\bm w} \in \R^{s+1}$ be the vector as above. Then, since $|{\bm r}-{\bm w}|_0<\delta\leq \delta_0$, $g_{\bm w}(x)$ is a degree $s$ separable polynomial satisfying $N_{g_{\bm w}}=\gamma$ by (S1) and also, by (S2), we have $\alpha_{\bm w}\leq \alpha_{\bm w}^{\prime}<\alpha_{\bm r}^{\prime}+\varepsilon_0=\rho_0^{\prime}<\xi$. Thus, by $(b)$ and the case (i), we have 
\begin{align*}
N_{f_{\bm w}(\xi ; x)}
=
\begin{cases}
\gamma+1 & \text{$n-s$ $:$ odd} \\
\gamma & \text{$n-s$ $:$ even, $r_s>0$} \\
\gamma+2 & \text{$n-s$ $:$ even, $r_s<0$},
\end{cases}
\end{align*}
which, by (S3), implies
\begin{align*}
N_{f_{\bm r}(\xi ; x)}
=
\begin{cases}
\gamma+1 & \text{$n-s$ $:$ odd} \\
\gamma & \text{$n-s$ $:$ even, $r_s>0$} \\
\gamma+2 & \text{$n-s$ $:$ even, $r_s<0$}.
\end{cases}
\end{align*}
Since $\xi$ is any real number such that $\xi>\rho_0^{\prime}$, this completes the proof of Claim and hence the proof of Theorem \ref{thm5.2}.
\end{proof}


In this section we want to investigate the irreducibility and the Galois group of the family of polynomials $f(t, x)$ over $\Q (t)$.  Galois groups over the transcendental extension $\Q(t)$ are of interest due to the Hilbert irreducibility theorem.  Computing Galois groups of polynomials over $\Q (t)$ we refer to \cite{smith} and \cite{mckay}.  

\begin{prop} 
Let $g(x)=\sum_{i=0}^s a_i x^i$ be a polynomial in $\R[x]$ such that $\D_g \neq 0$ and
\begin{equation} \label{eq-f}
f(t, x) = x^n + t \cdot g(x) 
\end{equation}
If $g(x)$ is totally complex, $(n-s)$ is even, and $a_s > 0 $ then $f(\beta, x)$ is totally complex for all $\beta > \max \{ \alpha \, | \, \D_{f, x} (\alpha) =0 \}$.
%
%
%
\end{prop}

\proof
To prove part ii) we have to show that $f(\beta, x)$ has no real roots for $\beta > \max \{ \alpha \, | \, \D_{f, x} (\alpha) =0 \}$. 
Since $g(x)$ is totally complex we have that $\gamma=0$. 
Form  Thm.~\ref{thm5.2} we have that   $N_{f, \beta} = \gamma$ or $\gamma+2$. The fact that $\beta > \max \{ \alpha \, | \, \D_{f, x} (\alpha) =0 \}$ implies that $N_{f, \beta} = \gamma=0$. Hence, $f(\beta , x)$ is totally complex. 
%
\qed




\begin{rem}
Polynomials in Eq.~\eqref{eq-f} for $s=1$ and $t=1$ has been treated in \cite{zarhin} while studying Mori trinomials.  It is shown there that the Galois group of $f(x)$  over $\Q$  is isomorphic to $S_n$; see \cite[Cor.~3.5]{zarhin} for details. 
\end{rem}

In general, if we let $K:=\Q (t, a_0, \dots , a_s)$ be the field of transcendental degree $s+1$, for $1 \leq s < n$,  then it is expected that $\Gal_K (f) \iso S_n$. 
  We have checked with Maple all values $3 \leq n\leq 8$ and $1\leq s< n$ and $\Gal_{\Q(t)} (f) \iso S_n$ in all cases. This would generalize Zarhin's result for a more general class of polynomials.

%
%
%
%
%
%




\bibliographystyle{amsalpha} 

\bibliography{ref}{}

\end{document}